\documentclass{article}
\usepackage[utf8]{inputenc}
\usepackage{amsmath}
\usepackage{graphicx}
\usepackage{geometry}
\usepackage{amssymb, gensymb}
\usepackage{amsthm}
\usepackage{mathtools}
\usepackage{mathrsfs}
\usepackage{indentfirst}
\usepackage{color}
\usepackage{datetime}
\usepackage{bm}

\usepackage{amsthm}

\newcounter{Definition}
\theoremstyle{definition}
\newtheorem{definition}[Definition]{Definition}
\newtheorem{example}[Definition]{Example}
\newtheorem{remark}[Definition]{Remark}
\theoremstyle{plain}
\newtheorem{theorem}[Definition]{Theorem}
\newtheorem{proposition}[Definition]{Proposition}
\newtheorem{lemma}[Definition]{Lemma}
\newtheorem{corollary}[Definition]{Corollary}

\newcommand{\bR}{{\mathbb{R}}}

\newcommand{\bE}{{\mathbb{E}}}
\newcommand{\bP}{{\mathbb{P}}}

\newcommand{\bC}{{\mathbb{C}}}

\newcommand{\bN}{{\mathbb{N}}}

\newcommand{\cF}{{\mathcal{F}}}
\newcommand{\cS}{{\mathcal{S}}}
\newcommand{\cB}{{\mathcal{B}}}

\title{Chernoff approximations of Feller semigroups in Riemannian manifolds}
\author{Sonia Mazzucchi,$^{1}$ Valter Moretti,$^1$  Ivan Remizov,$^2$ Oleg Smolyanov$^{3,4}$}

\date{November 2020}

\begin{document}

\maketitle

{\footnotesize
$^1$ Department of Mathematics, University of Trento and TIFPA-INFN,
via Sommarive 14,
I-38123 Povo (Trento) 
Italy

$^2$ National Research University Higher School of Economics, Russian Federation

$^3$ Lomonosov Moscow State University, Faculty of Mechanics and Mathematics,
Chair of Real and Functional Analysis, Laboratory of Infinite-Dimensional Analysis and Mathematical Physics 

$^4$ Moscow Institute of Physics and Technology

Emails: sonia.mazzucchi@unitn.it; valter.moretti@unitn.it; ivremizov@yandex.ru; smolyanov@yandex.ru
}

\begin{abstract}{
Chernoff approximations of Feller semigroups and the associated diffusion processes in Riemannian manifolds are studied.
The manifolds  are assumed to be of bounded geometry, thus including
 all compact manifolds and also a wide range of non-compact manifolds.
Sufficient conditions are established for a class of second order elliptic operators to generate a Feller semigroup on a (generally non-compact) manifold of bounded geometry. A construction of Chernoff approximations is presented for these Feller semigroups in terms of shift operators. This provides approximations of solutions to initial value problems for parabolic equations with variable coefficients on the manifold.  It  also yields weak convergence of a sequence of random walks on the manifolds to the diffusion processes associated with the elliptic generator. For parallelizable  manifolds this result is applied in particular to the representation of Brownian motion on the manifolds as limits of the corresponding random walks.}
\end{abstract}. 

{\small
\tableofcontents
}

\section{Introduction}

The relations between, on the one hand, the evolution equation and semigroup theory and, on the other hand, functional integration and the theory of stochastic processes is an extensively studied topic 
\cite{Bal,EthKur,IkeWat,Kal,Kol,RogWil} with  a long history.  Its roots can be  traced back to the pioneering papers by  Richard Feynman \cite{f-1948, f-1951}, who proposed  an heuristic representation of the solution to the Schr\"odinger equation in terms of limits of integrals over finite Cartesian powers of some spaces. Feynman's ideas inspired Marc Kac \cite{Kac1949}, who rigorously proved a representation of the solution of the heat equation in terms of an integral on the space of continuous paths with respect to the Wiener measure. This formula, which is nowadays known as the celebrated "Feynman-Kac formula", is the first and most famous example of the connections between parabolic equations associated with second order elliptic operators and stochastic processes. 
Remarkably, 
Feynman heuristically presented two mathematical constructions which are now associated with names of Trotter \cite{Tro}
 and Chernoff \cite{Chernoff}, who rigorously    proved them  much later. 
Trotter and Chernoff formulas provide approximations of evolution (semigroups) that, in several cases, pave the way for the proof of representation formulas of Feynman-Kac type. 

In the present paper,  new Chernoff approximations  are established  for  a particular class of Feller semigroups on a type of generally non-compact  Riemannian manifolds.
 In addition,  these  formulas are also proved to have a nice probabilistic interpretation on the said class of manifolds, since they allow  the proof of the weak convergence of a sequence of random walks on the manifold  to
 the diffusion process associated with the elliptic operator generating the said Feller semigroups. \\

 \noindent {\bf Literature on the subject}.  
 From a general perspective, this work refers to the theory of some strongly continuous semigroups of linear operators $(V(t))_{t\in \bR^+}$ on the Banach space $C_0(\mathcal{M})$
 of continuous real-valued functions vanishing at $\infty$ on a locally compact metric space $\mathcal{M}$. Such semigroups are  called  {\em Feller semigroups}. 
They are naturally associated with strong Markov stochastic processes $(X^x(t))_{t\in \bR^+}$ with values in the one-point compactification of $\mathcal{M}$ in such a way 
that the action of the operators $V(t)$ on a function $f\in C_0(\mathcal{M})$ can be represented in terms of the following formula
$$(V(t)f)(x)=\bE[f(X^x(t))], \qquad x\in \mathcal{M},\: t\in \bR^+\:.$$
$\bE$ is the expected value.  This paper considers the  concrete case where $\mathcal{M}$ is a smooth Riemannian manifold $M$ and the generator of the Feller semigroup when restricted to the space
 $C_c^\infty(M)$ of smooth functions with compact support is given by the second-order differential operator
\begin{equation}\label{eqL_0-1}
    (L_0f)(x)=\frac{1}{2}\sum _{k=1}^r(A_kA_kf)(x)+A_0f(x), \qquad x\in M,
\end{equation}  
where $A_k$, $k=0, \ldots, r$ are smooth vector fields. The stochastic processes associated with this particular kind of Feller semigroups are named {\em Feller-Dynkin diffusions}. They have continuous paths and can be constructed in terms of the (martingale) solution of stochastic differential equations of the form  \cite{Elw,Hsu,IkeWat,Wang}
\begin{equation}\label{SPDEi}dX(t)=\sum_{j=1}^r A_j(X(t))\circ dB^j(t) +A_0(X(t))dt.\end{equation}
This work in particular is devoted to the application of the Chernoff theorem (see theorem \ref{FormulaChernova} below) to the construction of an approximation formula for, on the one hand, the Feller semigroup and, on the other hand, the associated diffusion process and solutions to the evolution equation. This technique has been extensively implemented, e.g. in the study of 
Chernoff approximations of
Feller semigroups (and corresponding Feller processes) \cite{Butko-DM2010, Butko-IDAQP2012, Butko-FCAA2018, Butko-SD2018}, in the construction of solutions to evolution equations \cite{BauConGro,BoOrSa,Butko-FPM2006}, and in the  construction of the Wiener measure on  compact manifolds  \cite{Baer,SWW2007}
(see for overviews \cite{Butko-2019,SmHist,SmSchrHist}). Most of the results presented in literature are restricted to the case where either $M=\bR^d$ or $M$ is compact.  More general classes of $C^k$ (with $k=1,2 \ldots, \infty$ depending on the case) Riemannian manifolds were studied in \cite{Jorgensen,Pinsky,Li} (see also \cite{Manton} for an introductory overview of Brownian motion and diffusion processes on manifolds).
 In those papers, generally speaking, conditions are assumed about  (a) the existence of a  specific  cover of open sets with both uniform metric properties and uniform bounds 
on the vector fields $\{A_k\}_{k=0, \ldots, r}$ associated to the dynamical system \eqref{SPDEi} and (b) the validity of specific bounds on some curvatures. Under these conditions  it is possible to prove  the existence of Feller semigroups associated to the differential operator \eqref{eqL_0-1}  as well as the non explosion property of the associated process \cite{Li}. In \cite{Jorgensen,Pinsky} similar conditions allow proving the convergence of geodesic random walks  to the Brownian motion  on the manifold. 
A recent  remarkable book on semigroups on $L^2(M)$ (instead of $C_0(M)$) for generally non-compact manifolds $M$  
and the special case of Schr\"odinger-like  operators 
is \cite{Gu}. There, heat kernels  are extensively studied for Schr\"odinger-like  operators on Hermitian bundles 
on generally non-compact base manifolds, extending many known  results valid in $\bR^n$ to these geometric structures.\\

\noindent {\bf Results of this work}. 
In contrast to the quoted literature, the present work focuses on  continuous semigroups on $C_0(M)$ with generators of the
 form (\ref{eqL_0-1}) for 
the case of  a generic smooth Riemannian  manifolds $(M,g)$ 
{\em of bounded geometry}, also requiring  uniform boundedness properties of the involved vector fields for general 
elliptic operators  (\ref{eqL_0-1}).
Manifolds of bounded geometry are for instance $\bR^d$, compact manifolds, and  a wide class of {\em non-compact} 
manifolds that are also relevant in applications, like Lie groups and homogeneous manifolds.  The main results of this
 work follow.

(a) As the first result,  in  Section \ref{sec3} we show that if the vector fields $\{A_k\}_{k=0,\ldots,r}$ enjoy a property known as
 $C^\infty-$boundedness \cite{shubin}, then an extension of the differential operator $L_0$ in \eqref{eqL_0-1} 
is the generator of a Feller-Dynkin semigroup on $C_0(M)$, and we provide a family of operator cores.  This result paves the way for 
the proof of theorem \ref{teo3.1}, the second result of this paper, where a Chernoff approximation formula (Eq. \eqref{convergenceformula}) for the Feller semigroup in terms of a family of rather simple shift operators is presented. The idea of using shift operators instead of integral operators on $\mathbb{R}^d$ 
goes back to \cite{R-JMP2019, R-AMC2018, RS-MN2018,R-PotAnonlinefirst2018} and is now applied to manifolds for the first time. 
We also extend the described results to more general operators $L_0+c$, where $c\leq 0$ is a   bounded continuous scalar potential.

(b)  The probabilistic interpretation of the approximation formulas \eqref{S(t)} 
and \eqref{convergenceformula} in the case of $c=0$ is discussed in
Section \ref{SecProbabilisticInterpretation}.
 There, as the third main result, we show that it allows us to construct the diffusion process associated to the Feller semigroup in terms of a weak limit of a sequence of random walks on $M$.   
Several interesting convergence results for diffusion processes on manifolds can be found in literature, see e.g. \cite{Mol,DeGaMa, Jorgensen,Pinsky,Li}. It is worth mentioning the approximation schemes for the Wiener measure proposed in \cite{AndDri,Baer}, the proof of convergence of random walks to Brownian motion on sub-Riemannian manifolds \cite{GorLae} and the recent application of the notion of controlled rough path to Riemannian manifolds \cite{DriSem}. 
In contrast to the above mentioned results, in particular \cite{Jorgensen,Pinsky,Li}, where only geodesic paths are used in $M$ so that the 2nd order ODE are relevant,
in this paper we provide three different approximation schemes associated to {\em 1st order} differential equations of curves in $M$. 
These equations are the ones of integral lines of the aforementioned vector fields $\{A_k\}_{k=0,\ldots,r}$.
Indeed, the first approximation scheme involves a sequence of jump processes with random jumps along integral curves of the vector fields $\{A_k\}_{k=0, \ldots, r}$.  
Notice that more than one vector field is necessary to change the direction of the random walk when dealing with vector fields in $M$ instead of geodesics.
The second approximation scheme is a sequence of random walks with continuous piecewise geodesic paths. Finally, the third approximation scheme involves a sequence of random walks with continuous paths where the single steps are integral curves of the vector fields $\{A_k\}_{k=0, \ldots, r}$.

(c) These techniques are eventually applied
 in section \ref{sez5}
 to the Chernoff approximation of the specific case of the heat
 semigroup and the Brownian motion on {\em parallelizable} Riemannian manifolds. In this context we acheive the final results presented in this work.
As noted above, besides the traditional
 approximation of Brownian motion in terms of the weak limit of a sequence of random walks with piecewise geodesic 
paths (theorem \ref{teo43}), we  provide a new approximation result in terms of the limit of random walks with
 paths along the integral curves of a family of parallelizing vector fields (theorem \ref{teo44}).\\

\noindent {\bf Structure, notations, and conventions}.  The paper is organized as follows. Section \ref{sez2} presents 
some basic definitions and results on Feller semigroups, Chernoff approximations  and Riemannian
 geometry notions that are used throughout the paper. Section \ref{sec3}  presents the construction 
of the Feller semigroup and its Chernoff approximation. Section \ref{SecProbabilisticInterpretation} is 
devoted to the probabilistic interpretation of the Chernoff approximation formula and to the construction 
of three different sequences of random walks on $M$ converging weakly to the diffusion process associated to
 the Feller semigroup. Finally, section \ref{sez5} extends these results to the study of approximations of the heat 
semigroup and the Brownian motion on parallelizable manifolds of bounded geometry. The appendix contains the 
proofs of several technical propositions used in the main text.

 From now on the notation  $A\subset B$ includes the case $A=B$ and,
referring to a universe set $\mathcal{M}$,  if $A \subset \mathcal{M}$, then  $A^c := \mathcal{M} \setminus A$. 
Throughout the paper we adopt the definition $\bR^+ := [0,+\infty)$. 
If $M$ is a smooth manifold the symbol $C_c^\infty(M)$ denotes the complex space of smooth {\em compactly supported} complex-valued  functions on $M$.

An operator $A$ is always understood as a {\em linear} operator and its domain, denoted by $D(A)$, is always assumed to be a {\em linear subspace}. 
The symbol $\cB$ denotes
a Banach space over the field $\mathbb{C}$  or $\mathbb{R}$ and  $\mathscr{L}(\cB)$ denotes  the set of all bounded 
linear operators in $A: D(A) \to \cB$ with $D(A)=\cB$. 

If $A:D(A) \to \cB$
and $B: D(B) \to \cB$ are operators with $D(A),D(B) \subset \cB$, then
(i)  the domain of $A+B$ is defined as $D(A+B) := D(A)\cap D(B)$, 
(ii) the domain of $AB$ is defined as $D(AB):= \{x \in D(B) \:|\: Bx \in D(A)\}$, 
(iii) the domain of $aA$, with $a\in \bR$ or $\bC$,  is $D(aA) :=D(A)$ 
except for  $a=0$, where $D(0A)=\cB$; finally, $A \subset B$ means $D(A)\subset D(B)$ and $B|_{D(A)}=A$.

\section{Analytic and Geometric Preliminaries}\label{sez2}

We assume that the  reader is familiar with the theory of $C_0$-semigroups and we recall here just some basic definitions and results in order to fix the notation and the used terminology. 
We also recall some basic facts about the connection of the theory of $C_0$-semigroups and the theory of  random processes with particular emphasis on Feller semigroups and Feller processes.
Generally speaking, we shall focus attention only on the notions and the results which are strictly necessary to state and prove the results in the work. Details appear in the classical monographs \cite{EN1, Kol,EthKur, Bil, RogWil, IkeWat} and references therein.
Section \ref{Chernoffsec} contains some basic notions about {\em Chernoff-functions} \cite{Chernoff} which will be used in this work.
In sections \ref{riemmannian} and  \ref{Riemannian2} we shall remind the reader some basic notions of Riemannian geometry used throughout. Classical reference texts are \cite{Lee,DoCarmo,ONe, KobayashiNomizu}. Section \ref{bgeometry} introduces the basic notions and results on {\em manifolds of bounded geometry}. A recent review on the subject is \cite{DSS}.

\subsection{$C_0$-semigroups and evolution equations}\label{semgreveq}

\begin{definition}\label{semigrdef}
 A  mapping $V:\bR^+\to \mathscr{L}(\cB)$, is called a {\bf $C_0$-semigroup}, or {\bf a strongly continuous one-parameter semigroup} ({\bf of  bounded operators}) if it satisfies the following conditions,
	\begin{itemize}
\item[(1)] $V(0)= I$ the identity operator on $\mathcal{B}$,
	
\item[(2)]   $V(t+s)=V(t)V(s)$ if $t,s \in \bR_+$ ({\bf semigroup law}),
	
\item[(3)]  $\bR_+ \ni t \mapsto V(t)x$ is continuous for every $x\in \mathcal{B}$, i.e., $V$ is continuous in the {\em strong operator topology}.\hfill $\blacksquare$
	\end{itemize}
\end{definition}

As is well known \cite{EN1}, if $(V(t))_{t\geq 0}$ is a $C_0$-semigroup in Banach space $\cB$, then the set
 \begin{equation} D(L) := \left\{\varphi\in \cB \: \left|\: \exists \lim_{t\to +0}\frac{V(t)\varphi-\varphi}{t}\right.\right\}
\label{DL}\end{equation} is a dense linear subspace of $\cB$ invariant under the action of each $V(t)$, $t\geq 0$.
 The operator $L: D(L)\to \mathcal{B}$  $$L\varphi=\lim_{t\to +0}\frac{V(t)\varphi-\varphi}{t}\:, \quad \varphi \in D(L)$$
 is called the ({\bf infinitesimal}) {\bf generator}  of the $C_0$-semigroup $V$.
 The generator turns out to be  a closed linear operator that defines $V$ uniquely which, in turn, is denoted\footnote{As is well
 known, this notation is only formal in the general case even if in some situations it has a rigorous 
meaning in terms of norm-converging series if $L$ is bounded respectively spectral functional calculus in Hilbert spaces when $L$ is normal.} as $V(t)=e^{tL}$.

If $L: D(L) \to \mathcal{B}$ with $D(L) \subset \mathcal{B}$ is an operator, the problem of finding a function $u\colon \bR^+\to \mathcal{B}$ such that 
\begin{equation}\label{ACP1}
\left\{ \begin{array}{ll}
 \frac{d}{d t}u(t)= Lu(t); & t\geq 0,\\
 u(0)=u_0,\\
  \end{array} \right.
\end{equation}
is called the {\bf abstract Cauchy problem} (for the evolution equation) associated
to $L$. A function $u\colon \bR^+\to \mathcal{B}$ is called a {\bf classical solution} to abstract Cauchy problem (\ref{ACP1})
if, for every $t\geq 0$, the function $u$ has a continuous derivative (in the topology of $\cB$) $u'\colon \bR^+\to \mathcal{B}$, it holds 
$u(t)\in D(L)$ for $t\in \bR^+$, and (\ref{ACP1}) holds. The following fact can be found as Proposition 6.2 in \cite{EN1}, p. 145.

\begin{proposition}\label{ACPsol} 
 Let the operator $L: D(L) \to \mathcal{B}$
be the generator of a strongly continuous semigroup $(V(t))_{t\geq 0}$ in the Banach space $\mathcal{B}$.
Then, for every $u_0\in D(L)$ there is a unique classical solution to abstract Cauchy problem (\ref{ACP1}), which is
given by the formula $u(t)=V(t)u_0$.
\end{proposition}

\subsection{Feller semigroups and random processes}
$C_0$-semigroups are of particular interest because of their strong interplay with the theory of evolution equations, on the one hand, and with probability theory, on the other hand; from the probabilistic point of view the  so-called {\em Feller semigroups} \cite{Kol,EthKur} are particularly important. 

Let $\mathcal{M}$ be a locally-compact metric space. With the symbol   
$C(\mathcal{M})$ we denote the space of continuous functions $f: \mathcal{M} \to \bC$.
With $C_0(\mathcal{M})$ we shall denote   the Banach space of continuous functions  {\bf vanishing at $\infty$}, i.e. $$C_0(\mathcal{M}):=\{f \in C(\mathcal{M}) \:|\: \forall \varepsilon>0 \; \exists K \subset \mathcal{M} \hbox{ compact } |f(x)|<\varepsilon \; \forall x\in K^c \},$$  endowed with the $\|  \; \|_{\infty}$-norm.
If $\mathcal{M}$ is compact, it is natural to define $C_0(\mathcal{M}):=C(\mathcal{M})$.

A linear operator $U:C_0(\mathcal{M})\to C_0(\mathcal{M})$ is said to be {\bf positive} if $(Uf)(x)\geq 0$
for $x\in \mathcal{M}$ whenever $f\in C_0(\mathcal{M})$ and  $f(x)\geq 0$ if $x\in \mathcal{M}$. $U$ is said to be a {\bf contraction } if  $\|Uf\|\leq \|f\|$ for $f \in C_0(\mathcal{M})$.
\begin{definition}\label{semFellerdef} If $\mathcal{M}$ is a locally-compact metric space,
a strongly continuous semigroup made  of positive  contractions on $C_0(\mathcal{M})$ is called a {\bf Feller semigroup}. \hfill $\blacksquare$
\end{definition}
A crucial result is the following one  (theorem 2.2  Ch.4 in \cite{EthKur}):

\begin{theorem}\label{theoremFellergenerator} Let $\mathcal{M}$ be a locally compact metric space and $L_1 \colon D \to C_0(\mathcal{M})$ an operator with domain $D\subset C_0(\mathcal{M})$  subspace. 
$L_1$ is closable and its closure $L:= \overline{L_1}$ is the generator of Feller semigroup  if the following conditions are valid.
\begin{itemize}
\item[{\bf (a)}] $D$ is dense in $C_0(\mathcal{M})$,
\item[{\bf (b)}] $L_1$ satisfies the {\bf positive maximum principle}:  \\
\begin{equation}\hbox{for each } f\in D:
\hbox{ if }\sup_{x\in \mathcal{M}}f(x)=f(x_0)\geq 0  \hbox{ 
for } x_0\in \mathcal{M},\hbox{ 
then }(L_1f)(x_0)\leq 0\:,\label{PMP}\end{equation}
\item[{\bf (c)}] $Ran(L_1 - \lambda I)$ is dense in $C_0(\mathcal{M})$ for some $\lambda >0$.
\end{itemize}
\end{theorem}

\begin{remark}$\null$\\
{\bf (1)} Given a closed operator $L:D(L)\subset \cB\to \cB$ on a Banach space $\cB$, a dense subspace $D\subset D(L)$ is called a  {\bf core} for $L$ if $L|_{D}$ is closable  and $\overline{L|_{D}}=L$.\\
Theorem \ref{theoremFellergenerator} in fact yields the existence of the semigroup as well as a core for its generator.\\
  {\bf (2)} In this paper, $\mathcal{M}$ is a Riemannian manifold $(M,g)$. We will introduce  and use three types of operators: $L_0$ is always a differential operator defined on the whole $C^\infty(M)$, $L_1$ is its restriction to a suitable subspace $D_k$ satisfying the theorem above, $L= \overline{L_1}$ is the generator of the Feller semigroup.  $\hfill \blacksquare$
\end{remark}

By the {\em Riesz-Markov theorem},  it is possible to associate to any Feller semigroup $V$  a family $(p_t(x))_{t\geq 0, x\in \mathcal{M}}$ of positive Borel measures on $\mathcal{M}$  such that, for all $t\geq 0$,
$$(V(t)f)(x)=\int_\mathcal{M} f(y) p_t(x, dy), \qquad x\in \mathcal{M}$$
and, for all  $f\in C_0(\mathcal{M})$,
$$\lim_{x_n\to x}\int_\mathcal{M} f(y) p_t(x_n, dy)=\int_\mathcal{M} f(y) p_t(x, dy).$$
Moreover $p_t(x, \mathcal{M})\leq 1$. 

If all the measures of the family  $(p_t(x))_{t\geq 0, x\in \mathcal{M}}$ are {\em probability measures}, then the Feller semigroup is said {\bf conservative}.  In this case,  from the semigroup law, the family of probability measures satisfies the {\em Chapman-Kolmogorov} equation:
\begin{equation}\label{CK}p_{t+s}(x, A)=\int_\mathcal{M}p_t(y, A)p_s(x, dy), \qquad \hbox{for every Borel set } A\subset \mathcal{M}.\end{equation}
As a consequence,
given an arbitrary probability measure $\mu$ on the Borel $\sigma $-algebra $\cB(\mathcal{M})$ of $\mathcal{M}$, it is possible to construct a Markov process $(X^{\mu}_t)_{t\geq 0}$ with values in $\mathcal{M}$    with finite dimensional distributions 
\begin{equation}\label{fin-dim-distr}\bP(X^{\mu}_{t_1}\in A_1, \dots X^{\mu}_{t_n}\in A_n)=
\int 1_{A_1}(x_1)\cdots 1_{A_n}(x_n)p_{t_n-t_{n-1}}(x_{n-1}, dx_n)\cdots p_{t_1}(x_0, dx_1)d\mu (x_0),
\end{equation}
for $0\leq t_1\leq \dots \leq t_n$ and $A_1, ..., A_n \in \cB(\mathcal{M})$. The existence of the process is guaranteed by the  {\em Kolmogorov existence theorem} \cite{Bil}, the family of  measures \eqref{fin-dim-distr} being consistent due to the  Chapman-Kolmogorov identity \eqref{CK}.
In the general case,  it is still possible to define the associated Markov process $(X^{\mu}_t)_{t\geq 0}$ with values in  the 1-point compactification $\mathcal{M}':=\mathcal{M}\cup \partial$ of $\mathcal{M}$ and 
 the process enjoys the {\em strong Markov property} \cite{RogWil}. If $X_s=\partial $ $\forall s\geq t$ whenever either $X_{t^-}=\partial $ or $X_t=\partial$,  then these processes are called {\em Feller-Dynkin} (FD-) processes. The random variable 
$$\xi:=\inf\{t\in \bR^+ | X_t=\partial\}$$
is called {\em lifetime} or {\em explosion time} of the process.
In fact, if the Feller semigroup is conservative then $\xi=+\infty$ almost surely, hence the   FD-process can be thought as a stochastic process with values in $\mathcal{M}$ instead of $\mathcal{M}'$ and it is called conservative.

By \eqref{fin-dim-distr} the action of the semigroup admits the following probabilistic representation
\begin{equation}(V(t)f)(x)=\bE[f(X_t^x)],\qquad x\in \mathcal{M}, \label{prorep}\end{equation}
where $X_t^x$ is the aforementioned Markov process with initial distribution $\mu=\delta_x$, the Dirac measure concentrated at $x\in \mathcal{M}$.

An important class of FD-processes are the {\em diffusions}, also called Feller-Dynkin diffusions \cite{IkeWat,RogWil}. They are defined as FD-processes with continuous paths up to the explosion time.  The generator $L$ of the associated semigroup is a local operator with a domain that includes the set  of smooth functions with compact support and $L$ satisfies the maximum principle \eqref{PMP} there. If  $x\in \mathcal{M}$ and $(X^x_t)$ is the diffusion process starting at $x$, then its law $P^x$ is a probability measure on the metric space $C(\bR^+, \mathcal{M})$ of continuous paths on $\mathcal{M}$ or, more generally in the case of explosion, on  $C(\bR^+, \mathcal{M}')$. The family $\{P^x\}_{x\in \mathcal{M}}$ is called a {\em system of diffusion measures}.\\ 
 In the case where the state space $\mathcal{M}$ of the Feller-Dynkin diffusion is  $\bR^d$, it is well known (see e.g. \cite{RogWil,Kol}) that
  the restriction of $L$ to $C^\infty_c(\bR^d)$ is a second-order elliptic operator of the form
 \begin{equation}\label{generator1}(L_0f)(x)=\sum_{i,j}a^{ij}(x)\frac{\partial^2 f}{\partial x^i\partial x^j}(x)+\sum_{j}b^{j}(x)\frac{\partial f}{\partial x^j}(x)+ c(x)f(x), \qquad x\in \bR^d,\quad f \in C_c^\infty(\bR^d) \:.\end{equation}
where $a^{ij}, b^j$, $c$,  $i,j=1,\dots,d$,  are real-valued  continuous  functions, $c\leq 0$ and  the matrix  of coefficients  $a^{ij}(x)$ is symmetric and non-negative definite.
The corresponding semigroup $V$  provides a classical solution of the Cauchy problem (in the above semigroup sense) for $u_0\in C_c^\infty(\bR^d)$,
\begin{equation}\label{CPee}
\left\{ \begin{array}{ll}
u'_t(t,x)=Lu(t,x) \ \mathrm{ for }\ t>0, x\in \bR^d\\
u(0,x)=u_0(x)\ \mathrm{ for } \ x\in \bR^d
\end{array} \right.
\end{equation}

Actually, by formula \eqref{prorep}, the function $ u:\bR^+\times \bR^d\to \bR$ admits the probabilistic representation formula $u(t,x)=\bE[u_0(X_t^x)]$. 

Conversely, given globally Lipschitz maps $\sigma _k^i: \bR^d\to \bR$ and $b^i:\bR^d\to \bR$ and setting $a^{ij}=\sum_k\sigma _k^i\sigma _k^j$, it is possible to prove that there exists a Feller semigroup whose generator restricted to $C^\infty_c(\bR^d)$ has the form \eqref{generator1} with $c=0$. The associated diffusion process is constructed in terms of the so called martingale solution of the stochastic differential equation 
\begin{equation}\label{SPDE1}dX^i_t=\sum_{k=1}^d \sigma _k^i(X_t)dB_t^k+b^i(X_t)dt,\end{equation}
where $(B_t)_{t\in \bR^+}$, is a $d$-dimensional Brownian motion. For an extended discussion of this topic see, e.g. \cite{RogWil,IkeWat}.

\subsection{Chernoff approximations for $C_0$-semigroups} \label{Chernoffsec}
Here we recall {\em Chernoff's theorem}   \cite{Chernoff,EN1,BS} which provides approximation method for $C_0$-semigroups 
on Banach space in terms of suitable operator valued functions.

\begin{theorem}[The Chernoff theorem]\label{FormulaChernova}
Let $(e^{tL})_{t\geq 0}$ be a $C_0$-semigroup on a Banach space $\cB$ with generator  $L: D(L) \to \cB$ and let $S:\bR^+\to \mathscr{L}(\cB)$ be a map satisfying the following conditions:
\begin{enumerate}
    \item There exists $\omega\in\mathbb{R}$ such that $\|S(t)\|\leq e^{\omega t}$ for all $t\geq 0$;
    \item The function $S$ is  continuous in the strong topology in  $\mathscr{L}(\cB)$;
    \item $S(0) = I$, i.e., $S(0)f=f$ for every $f \in \cB$;
    \item There exists a linear subspace $\mathcal{D} \subset D(L)$ that is a core for the operator $L : D(L) \to \cB$   and such that    $\lim_{t \to 0}(S(t)f-f-tLf)/t=0$ for each $f \in \mathcal{D}$.
\end{enumerate}
Then the following holds:
\begin{equation}\label{Chest}
\lim_{n\to\infty}\sup_{t\in[0,T]}\left\|S(t/n)^nf - e^{tL}f\right\|=0,\quad \mbox{for every $f\in \cB$ and every $T>0$,}
\end{equation}
where $S(t/n)^n$ is a composition of $n$ copies of the linear bounded operator $S(t/n)$.
\end{theorem}

\begin{remark}{ 
Let $(e^{tL})_{t\geq 0}$ be a $C_0$-semigroup on a Banach space $\cB$ with generator  $L: D(L) \to \cB$ and let $S:\bR^+\to \mathscr{L}(\cB)$ be a map satisfying formula \eqref{Chest} then:
\begin{itemize} 
\item[(a)]
$S$ is called a {\bf  Chernoff function} for operator $L$ or 
  {\bf Chernoff-equivalent} to $C_0$-semigroup $(e^{tL})_{t\geq 0}$ \cite{STT}.
  
\item[(b)]The expression $S(t/n)^nf$ is called a \textit{\bf Chernoff approximation expression} for $e^{tL}f$.

\item[(c)] The $\mathcal{B}$-valued function
$$U(t):=\lim_{n\to\infty}S(t/n)^nu_0=e^{tL}u_0$$ is the classical solution of the  Cauchy problem (\ref{ACP1}) due to Proposition \ref{ACPsol} and Theorem \ref{FormulaChernova} if $u_0\in D(L)$, so
Chernoff approximation expressions become approximations to the solution with respect to norm in $\mathcal{B}$.
\end{itemize}
}\end{remark}

A  definition of Chernoff equivalence and Chernoff function was suggested in  2002 \cite{STT} and developed in \cite{SWWdan,SWWcan,SWW2007,SmHist,SmSchrHist, R5}.  New wording was proposed in  \cite{R1, R-AMC2018}. Every $C_0$-semigroup $S(t)=e^{tL}$ is a Chernoff function for its generator $L$, actually it is the only one
 Chernoff function which has a semigroup composition property. 
 Also there are other statements known as Chernoff-type theorems and they produce different notions of Chernoff function. Here we will not give an overview of this topic. We just fix one version of the Chernoff theorem, one definition of Chernoff function and work with it.

\subsection{Structures on Riemannian manifolds}\label{riemmannian}
In this section we recall some  general notions of Riemannian geometry. For more details we refer to 
\cite{Lee,DoCarmo,ONe, KobayashiNomizu}.
 Let  $(M,g)$ be  a  smooth (i.e., $C^\infty$) Riemannian manifold, which we will always assume to be connected, Hausdorff, and 2nd countable. 
The {\bf Riemannian 
distance} of $p,q \in M$ is defined as 
\begin{equation} \label{defd}d_{(M,g)}(p,q) = \inf_{\gamma \in C_{p,q}} L_g(\gamma)\:.\end{equation}
Above, $C_{p,q}$ is the set of the smooth curves $\gamma : [a,b] \to M$ with $\gamma(a)=p$ and  $\gamma(b)=q$  ($a<b$ depend on $\gamma$) and 
$$L_g(\gamma) := \int_a^b \|\dot{\gamma}(t)\|_g dt\:,$$
-- where $\dot{\gamma}$ is the tangent vector to $\gamma$ and $\|\dot{\gamma}(t)\|_g=\sqrt{g_{\gamma(t)}(\dot\gamma(t),\dot\gamma(t))}$ its standard $g$-norm (see below) --
 is the {\bf length} of the curve $\gamma$ computed with respect to $g$.
 The Riemannian distance makes $M$ a metrical space whose metrical topology 
coincides with the original  topology of $M$ as topological manifold.

If $p\in M$ and $U_p \subset T_pM$ is a sufficiently small open neighborhood of the origin $0\in T_pM$,
the {\bf exponential map} at $p$, denoted by  $\exp_p :U_p \to M$, is the map associating $v\in U_p$ with $\sigma(1,p,v)$, where 
$[0,1] \ni s \mapsto \sigma(s,p,v) \in  M$ is the restriction to $[0,1]$ of the maximal $g$-geodesic in $M$ starting from $p$, at $s=0$,
 with initial tangent vector $v$.  It is known that if $U_p$ is sufficiently small, $\exp_p$ is a diffeomorphism 
from   $U_p\subset T_pM$  onto the open neighborhood  $V_p := \exp_p(U_p)  \subset M$ of $p$. 
Furthermore,  such $V_p$ can be chosen to be an {\em open $d_{(M,g)}$-metric ball} $V_p= B^{(M,g)}_r(p)$ of 
sufficiently small radius $r>0$ (in this case $U_p$ will be the open ball in $T_pM$ with radius $r$).

 With the said choice of  $B^{(M,g)}_r(p)$,
  if $N:= \{ e_1,\ldots e_d\}$ is a $g$-orthonormal basis of $T_pM$, we can construct a  bijective  map  denoted by  $\exp^{-1}_{p,N}:  B^{(M,g)}_r(p)\to B_r(0) \subset \bR^d$ as:
$$\exp^{-1}_{p,N}:  B^{(M,g)}_r(p) \ni q \mapsto (y^1(q), \ldots, y^d(q)) \in B_r(0) \subset \bR^d\quad \mbox{where $\sum_{j=1}^d y^j(q) e_j =
 \exp_p^{-1}(q)$.}$$ This map  is smooth with its inverse and its image (i.e., the coordinate representation of the open neighborhood of
 the origin of $T_pM$  previously denoted 
by $U_p$) 
is  a standard ball $B_r(0) \subset \bR^d$ centered at the origin with the same  radius $r$ as $B_r^{(M,g)}(p)$. The pair
$(B^{(M,g)}_r(p), \exp^{-1}_{p,N})$ is called  a (local)  {\bf normal Riemannian chart centered on} $p$ and the coordinates $y^1,\ldots,y^d$, {\bf Riemannian coordinates centered on} $p$.

It turns out that, referring to this coordinate patch,
\begin{itemize} \item[(a)]  the components at $y\in B_r(0)$ of the metric and its inverse respectively
 satisfy $g_{ab}(0)= \delta_{ab}$ and   $g^{ab}(0)= \delta^{ab}$ for $a,b = 1,\ldots, d$;
\item[(b)]  the {\bf Levi-Civita connection coefficients}  (see (\ref{CCLC}) below) $\Gamma^c_{ab}(y)$ associated to metric satisfy 
$\Gamma^c_{ab}(0)= 0$ and it also holds  $\frac{\partial g_{ab}}{\partial y^c}|_0 = \frac{\partial g^{ab}}{\partial y^c}|_0 =0$   for $a,b, c = 1,\ldots, d$; 
\item[(c)]  the $\bR^d$-Euclidean norm in $B_r(0)$ coincides with the distance from 
$p$ in the following sense: \begin{equation} \|y\| = d_{(M,g)}\left(\exp_p\left( \sum_{j=1}^d y^j(q) e_j \right), p\right);\label{distanceG}\end{equation}
\item[(d)]  there is a unique geodesic segment $\gamma$  joining $p$ and $q\in  B_r^{(M,g)}(p)$ and completely included in $B_r^{(M,g)}(p)$. In Riemannian coordinates centered on $p$, it coincides with the $\bR^d$ segment  joining the origin to $(y^1(q), \ldots, y^d(q))$.
The length $L_g(\gamma)$ is
 $d_{(M,g)}(p,q)$.
\end{itemize}

 $(M,g)$ is said to be {\bf geodesically complete} if all geodesics are  defined for all values of their affine  parameter in $\bR$. Another way to say the same is that the exponential map $\exp_x$, for every given $x\in M$, is defined on the whole $T_xM$ (even if this does not imply that it defines a diffeomorphism on the whole $T_xM$). 
The celebrated {\em Hopf-Rinow theorem} proves that geodesical
 competeness is equivalent to the fact that $M$ is complete as a metric space with respect to $d_{(M,g)}$. In turn this is equivalent
 to the fact that closed bounded (with respect to the geodesical distance)  subsets of $M$ are compact. Finally for 
geodesically complete manifolds, every pair $p,q\in M$ admits a (not necessaily unique) geodesic joining
 them and the length of this  geodesic segment coincides with  $d_{(M,g)}(p,q)$, since the said geodesic
minimizes the length of the curves joining the points.

The {\bf injectivity radius at $p\in M$}, denoted by $I_{(M,g)}(p) \in \bR^+$,  is the supremum of the set of radii $r$ of the open ball
 $B^{(M,g)}_r(p) \subset M$  such that $(B^{(M,g)}_r(p), \exp^{-1}_{p,N})$ is a   normal Riemannian chart centered at $p$ for 
an orthonormal basis $N$ of $T_pM$ (it does not depend on $N$).  The {\bf injectivity radius} of $(M,g)$ is
$$I_{(M,g)} := \inf_{p\in M} I_{(M,g)}(p)\:.$$
\begin{remark}\label{compact-finiteI}
Compact smooth Riemannian  manifolds in particular  have always strictiy positive injectivity radius as the reader easily proves. \hfill $\blacksquare$
\end{remark}
Strictly positivity of the injectivity radius has several important consequences, the following one in particular.
 
\begin{lemma}\label{lemmaC} If $(M,g)$ is a connected smooth manifold  with  strictly positive injectivity radius, then $(M,g)$ is geodesically complete and all closed bounded sets  are compact. 
\end{lemma}
\begin{proof} See the appendix. \end{proof}

\subsection{Manifolds of bounded geometry}\label{bgeometry}
For future use, we introduce the definition of manifold $(M,g)$ {\em of bounded geometry}. This is a class of Riemannian
 manifolds where, in particular,  the thesis of Lemma \ref{lemmaC} is valid. See \cite{DSS} for a recent extended review
 and \cite{Kor,shubin} for a summary of notions and results used in this paper.
 Roughly speaking (see remark \ref{remBG} below), bounded geometry means that, on the one hand, every point $p\in M$ on the manifold there is  a geodesical ball $B^{(M,g)}_r(p)$ covered by Riemannian coordinates  centered on $p$ of  radius $r>0$ independent of $p$. On the other hand, there are uniform bounds on all derivatives of the component of 
the metric in the said Riemannian coordinates in $B^{(M,g)}_r(p)$ independent of $p$. Here is the formal definition. 
 \begin{definition}\label{defboundedgeom}
  A connected smooth Riemannian manifold $(M,g)$ is said
 {\bf of bounded geometry} if $(M,g)$ has strictly positive injectivity radius and for some constants $c_k<+\infty$, $k=0,1,\ldots$
$$\| \|\nabla^{(g)k}  R\|_g\|_\infty \leq c_k \:, \quad k=0,1,\ldots\:. $$
\hfill $\blacksquare$
 \end{definition}

Above and henceforth, $\nabla^{(g)}$ indicates the {\em covariant derivative} of the {\em Levi-Civita connection} associated to $g$, $R$ indicates the {\em Riemannian curvature tensor} and $\|\cdot\|_g$ denotes the natural point-wise norm associated to the metric $g$ acting on smooth tensor fields of a given  order
 (order $(1, 3+k)$ concerning $\nabla^{(g)k} R$). For instance,  if $T$ is a smooth tensor field of order $(n,m)$, so that their components at $q\in M$ 
in coordinates $y^1,\ldots, y^d$ around $q$ are 
${T^{a_1\cdots a_n}}_{b_1\cdots b_m}(y(q))$, we have
\begin{multline}\|T(q)\|^2_g =\sum_{a_1,\dots, a_n, b_1,\dots, b_n,  c_1,\dots, c_n, d_1,\dots, d_n } g_{a_1c_1}(y(q)) \cdots g_{a_nc_n}(y(q)) g^{b_1d_1}(y(q))\cdots g^{b_md_m}(y(q))\\ {T^{a_1\cdots a_n}}_{b_1\cdots b_m}(y(q))
 {T^{c_1\cdots c_n}}_{d_1\cdots d_m}(y(q))\:.
 \end{multline}

\begin{example} From the definition above, the following manifolds in particular are of bounded geometry (Example 2.1 in \cite{DSS,shubin}):
\begin{itemize}
\item[(i)] every smooth compact Riemannian manifold;

\item[(ii)] $\bR^m$ equipped with its natural metric;

\item[(iii)] every smooth Riemannian locally flat manifold with strictly positive injectivity radius;

\item[(iv)]  some classical manifolds as the $m$-dimensional hyperbolic space (the unit ball $B_1(0)$ in $\bR^m$
 equipped with the {\em Poincar\'e disk metric});
\item[(v)]  Homogeneous manifolds with invariant metric;
\item[(vi)]  covering manifolds of compact manifolds with a Riemannian
metric  which is lifted from the base manifold.\hfill $\blacksquare$
\end{itemize}
\end{example}

Another crucial  feature of a smooth Riemannian manifolds of bounded geometry is the one that  follows \cite{DSS}.  For every given $r \in (0, I_{(M,g)}]$, there is a sequence of finite constants 
$C^{(r)}_k \in \bR^+$, $k=0,1,2,\ldots$ and a constant $c^{(r)}>0$ such that 
\begin{equation}
\det [g_{ab}(y)] 
 \geq c^{(r)}\:, \mbox{ if $y\in B_r(0)$} \quad \mbox{and} \quad \max_{|\alpha|\leq k} \| \partial_y^\alpha
 g_{ab}(y)\|^{(B_r(0))}_\infty \leq C_k^{(r)}\:, \quad a,b = 1,\ldots, d \label{estimateg}
\end{equation}
where $y^1,\ldots, y^n$ are the coordinates of every normal Riemannian chart  with domain  $B_r^{(M,g)}(p)$ centered at $p\in M$ and $g_{ab}(y)$ are the components of the metric in that local coordinate system.
We stress that the constant $C_k$ do {\em not} depend on $p$ and {\em all domains have the same geodesical radius $r$}.

From (\ref{estimateg}) 
taking advantage of the Kramer rule to compute the element $g^{ab}(y)$ of the inverse of the matrix of the coefficients $g_{ab}(y)$, as well as recursively using the identity
$$\frac{\partial g^{ab}}{\partial y^i}= -\sum_{c,d}g^{ac}g^{bd} \frac{\partial g_{cd}}{\partial y^i}\:,$$
it easily arises the existence of another sequence 
of finite constants 
$H^{(r)}_k \in \bR^+$, $k=0,1,2,\ldots$  such that 
\begin{equation}
 \max_{|\alpha|\leq k} \| \partial_y^\alpha g^{ab}(y)\|^{(B_r(0))}_\infty \leq H_k^{(r)}\:, \quad a,b = 1,\ldots, d \label{estimateg2}
\end{equation}
where, as above, $y^1,\ldots, y^n$ are the coordinates of every normal Riemannian chart  with domain  $B_r^{(M,g)}(p)$ centered at $p\in M$ of radius $r \in (0, I_{(M,g)}]$.  

Finally, referring to Levi-Civita's connection coefficients \begin{equation}\label{CCLC}
 \Gamma^a_{bc}(y):= \frac{1}{2}\sum_{d}g^{ad}(y)\left( \partial_{y^c}g_{bd} +\partial_{y^b} g_{dc} -\partial_{y^d} g_{bc}\right)\:,\end{equation}
from the above pair of results, we obtain the  existence of another sequence 
of finite constants 
$J^{(r)}_k \in \bR^+$, $k=0,1,2,\ldots$  such that 
\begin{equation}
 \max_{|\alpha|\leq k} \| \partial_y^\alpha \Gamma^{a}_{bc} (y)\|^{(B_r(0))}_\infty \leq J_k^{(r)}\:, \quad a,b,c = 1,\ldots, d \label{estimateg3}
\end{equation}
valid in every  normal Riemannian chart  around every $p\in M$ as before defined on a metric ball of radius $r \in (0, I_{(M,g)}]$
 with center $p$.  

\begin{remark}\label{remBG}
We observe {\em en passant} that if $(M,g)$ has strictly positive injectivity radius and satisfies (\ref{estimateg})
for a given $r \in (0, I_{(M,g)})$ -- so that it also satisfies  (\ref{estimateg2}) and  (\ref{estimateg3}) -- it is necessarily of bounded geometry, 
just in view of the polynomial expression  in components of the Riemann tensor in terms of $\Gamma_{ab}^c$ and their first derivatives.\hfill $\blacksquare$
\end{remark}
\subsection{Completeness of vector fields}\label{Riemannian2}
Let $M$ be a general smooth manifold.
As a  vector field $A$ on $M$ is a  map $A:M\to TM$, we use the notation $A(p)\in T_pM$.

Assuming that $A$ is smooth,  let us consider the Cauchy problem 
\begin{equation}\label{CP-M} \left\{
\begin{array}{l}\dot \gamma (s)=A(\gamma(s))\\
\gamma (t_0)=x
\end{array}
\right.\end{equation}
A solution $\gamma :(\alpha, \beta)\to M$ of \eqref{CP-M} is called {\bf maximal} if it is not the proper restriction of any other solution of  \eqref{CP-M}. 
By the uniqueness of local solution of the Cauchy problem 
 \cite{ONe} there exists only one maximal solution $\gamma$ of  \eqref{CP-M} and any other solution is one of its restrictions. $\gamma$   is called the {\bf  maximal integral curve of $A$ starting at $x$}.
A smooth vector field $A$ on the smooth manifold $M$  is said to be {\bf complete} [\cite{ONe}, p. 51] if each of its maximal integral curves is defined on the entire real line. 
We finally quote an elementary but crucial technical results whose proof is incuded for completeness in the appendix.
\begin{lemma}\label{lemma-complete-solutions}
  Let $(M,g)$ be a connected geodesically complete Riemannian manifold. Let $A$ be a smooth vector field such that
  \begin{equation}\label{assAalpha}
     \| \|A\|_g\|_\infty<+\infty\,.
  \end{equation}
  Then the maximal solutions of \begin{equation}\label{CPVF}\frac{d}{dt}\gamma (t)=A(\gamma (t))
  \end{equation} are complete.
\end{lemma}

\begin{remark}\label{remcompleteness}
 The thesis of the lemma is automatically satisfied for a smooth field in the case of compact manifolds (for instance as consequence of remark \ref{compact-finiteI} and  lemma \ref{lemmaC}, but the result is elementary and valid also in absence of metric $g$). Yet, assuming that  $A$ is {\em $C^\infty$-bounded} (see definition \ref{def22} below), the remaining hypotheses are true  for  manifolds of bounded geometry,  as a consequence of lemma \ref{lemmaC}. Hence the thesis of lemma \ref{lemma-complete-solutions} is valid also in this case. \hfill $\blacksquare$
\end{remark}

\section{Feller semigroups and Chernoff approximations for diffusions on Riemannian manifolds}\label{sec3}
 This section is devoted to the study of diffusions on Riemannian manifolds $(M,g)$ of bounded geometry. We consider second-order elliptic operators $L_0: C^\infty(M) \to C^\infty(M)$ of the form \eqref{Hoperator0} proving that they admit an extension $L:D(L)\subset C_0(M)\to C_0(M)$ that generates a Feller semigroup $(e^{tL})_{t\in \bR^+}$ on $C_0(M)$. We also provide a family of operator-cores for $L$. This result is finally applied in section \ref{sez3.3} to the construction of Chernoff approximations for the semigroup $(e^{tL})_{t\in \bR^+}$ in terms of a family of shift operators.
\subsection{Relevant operators and subspaces of $C_0(M)$}
Let $(M,g)$ be a $d$-dimensional $C^\infty$ connected Riemannian manifold which we also  assume to be  geodesically complete.
Let $\{A_k\}_{k=0, 1,...,r}$ be a family of $C^\infty$ vector fields on $M$. We start by considering   the second order differential operator $L_0: C^\infty(M) \to C^\infty(M)$
\begin{equation}\label{Hoperator0}
(L_0f)(x):=\frac{1}{2}\sum_{k =1}^rA_k(A_k f)(x)+(A_0f)(x), \qquad x\in M\:, \quad f \in C^\infty(M)
\end{equation}
In every  local coordinate neighbourhood $U$ containing $x$, if $\sigma^i_k(x)$ are the components of the vector $A_k$, the operator $L_0$  can be represented by the differential operator 
\begin{equation}\label{op-coordinates}
(L_0f)(x)=\frac{1}{2} \sum_{i, j}a^{ij} (x)\frac{\partial ^2}{\partial x^i\partial x^j} f(x)+\sum_{i}b^i(x)\frac{\partial }{\partial x^i}f(x), \qquad x\in U,\end{equation}
with $b^i(x)=\sigma_0^i(x)+\frac{1}{2}\sum_{j,k}\sigma^j_k(x)\frac{\partial}{\partial x^j}\sigma ^i_k(x)$ and $a^{ij} (x)=\sum _k\sigma _k^i(x)\sigma_k^j(x)$ are the entries of a
 positive semidefinite matrix.

$(L_0 + c): C^\infty(M) \to C(M)$ with $L_0$ taking the form  (\ref{op-coordinates})
 in every coordinate patch, and $c \in C(M)$ used as a multiplicative operator,
  is said to be 
{\bf elliptic at $x\in M$}  if the matrix of coefficients $a^{ij}(x)$ is positive semidefinite and  {\em non-singular} in every local coordinate system of $M$ around $x$. If this condition holds for every $x\in M$, then $L_0+c$ 
is said to be {\bf elliptic}.
It is easy to see that $L_0+c$ is elliptic   if the matrices of coefficients $a^{ij}$ are positive semidefinite and non-singular  in every chart of  an atlas of $M$.
\begin{remark} If $A_k$, $k=0,\ldots r$, are smooth vector fields on the smooth manifold $M$, then the 2nd order operator
 $L_0+c:= \frac{1}{2}\sum_{i=1}^r A_iA_i + A_0 +c$ is elliptic at $p\in M$ if and only if the vector fields $A_k$, with $k=1,\ldots, r$, define a  set of generators of $T_pM$.
 (In particular, ellipticity requires $r \geq d:= \dim M$ necessarily). 
In order to prove this fact, it is sufficient to notice that $a^{ij} (p)=\sum _k\sigma _k^i(p)\sigma_k^j(p)$ is automatically positive semidefinite,  hence ellipticity  at $p$ is equivalent to 
\begin{equation}\label{ellipticp}\sum_{k=1}^r   \langle \sigma_k(p), \omega \rangle    \sigma_k(p)  =0 \quad \mbox{iff}\quad \omega=0 \quad \mbox{when 
$\omega \in T_p^*M$}\:,\end{equation}
 where $\langle \cdot, \cdot \rangle$ is the standard pairing on $T_pM \times T^*_pM$ and (\ref{ellipticp}) holds iff $\{A_j(p)\}_{j=1,\ldots,d}$ generates $T_pM$. $\hfill \blacksquare$
\end{remark}

$L_0+c$ is said {\bf uniformly elliptic} (with respect to the metric $g$) if there is a costant  $C>0$ such that
$$\sum_{i,j=1}^n a^{ij}(x) \xi_i \xi_j \geq C\sum_{i,j=1}^n g^{ij}(x) \xi_i \xi_j  \quad \mbox{for every $\xi_k\in \bR$, $k=1,\ldots, d$, and every coordinate patch over $M$.}$$
It is easy to see that if the condition above is true for the local charts of an atlas of $M$ and a given $C>0$, then it is true for all local charts of $M$ for the same $C$. 
\begin{remark}
It is elementary to prove that, if $L_0+c$ is elliptic and $M$ is compact, then $L_0+c$ is uniformly elliptic. \hfill $\blacksquare$
\end{remark}

In general,  the space  $C_c^\infty(M)$ is dense in $C_0(M)$.

\begin{proposition}\label{corollaryD}  If $M$ is a smooth manifold, then $C_c^\infty(M)$ 
is dense in $C_0(M)$ in the norm $||\cdot||_\infty$.
\end{proposition}

\begin{proof}
See the appendix.
\end{proof}

\subsection{Generators of Feller semigroups on Riemannian manifolds}
This section is devoted to the construction of generators of Feller semigroup on $C_0(M)$ as well as to the description of their cores. In the following we shall always assume that $(M,g) $ is a smooth manifold of bounded geometry. We start by giving the definition of some relevant subspaces of smooth functions.
\begin{definition}
Let $(M,g)$ be a manifold of bounded geometry. A  function $f:M\to \bR$ is said {\bf $C^k$-bounded} if $f\in C^k(M)  $ and if  for every $r_0 \in (0, I_{(M,g)})$ and every  multiindex $\alpha$, with $|\alpha |\leq k$ there is a constant   $C_\alpha <\infty$ such that $|\partial ^\alpha _xf(x)|\leq C_\alpha <+\infty$  in every local Riemannian chart $(B_{r_0}^{(M,g)}, \exp^{-1}_{p,N})$ centered at every $p\in M$. \\
A function $f:M\to \bR$ is said {\bf $C^\infty$-bounded} if $f$ is  $C^k$-bounded for any $k\geq 0$.\\
The space of $C^k$-bounded functions on $M$ is denoted  with the symbol $C^k_b(M)$ for $k=0,1,\ldots, \infty$. \hfill $\blacksquare$
\end{definition}

\begin{remark}
It is easy to prove \cite{shubin} that $f\in C^k (M)$ is $C^k$-bounded iff there exists a constant $C< +\infty$ such that the covariant derivative  $\| \nabla ^k f\|_\infty<C$.  
\end{remark}

Let us consider the operator $L_0$ (\ref{Hoperator0}) and define  $L_1$ as its restriction to one of the linear subspaces $D_k\subset C_0(M)$ 
\begin{equation}\label{domainD}
    D_k:=\{f \in C_0(M)\cap C^\infty(M) \cap  C_b^k (M) \:\: |\: \: L_0f\in C_0(M)\} \quad \mbox{for $k=0,1,\ldots, \infty$.}
\end{equation}
 Each
$D_k$ is non-trivial and dense in $C_0(M) $ since $C_c^\infty(M) \subset D_k $ and by proposition \ref{corollaryD}. Actually, for every given $k$,  $L_1$
  satisfies hypotheses (a) and (b) of theorem \ref{theoremFellergenerator}, the latter can be  trivially proved by direct inspection.
If  we are able to prove that also hypothesis (c) of theorem \ref{theoremFellergenerator} is fulfilled (there  exists a $\lambda >0$ such that $Ran (L_1-\lambda I)$ is dense in $C_0(M)$), 
then theorem \ref{theoremFellergenerator} proves that $L:= \overline{L_1}$  is the generator of a Feller semigroup $(V(t))_{t\geq 0}$ on $C_0(M)$.  

\begin{remark}\label{exampleRd2} In the case  $M=\bR^d$ and the coefficients $a^{ij}, b^j$ of the differential operator \eqref{generator1}  are bounded and globally Lipschitz (their smoothness is guaranteed by the assumptions that the vector fields $A_k$ are smooth), probabilistic arguments \cite{RogWil} provide the existence of a Feller semigroup. The associated diffusion process is constructed in terms of the martingale solution of the stochastic PDE \eqref{SPDE1}.   In this case the representation formula \eqref{prorep} allows to prove that the generator restricted on the space $C^\infty _c(\bR^d)$ is actually  given by the second order operator \eqref{generator1}. \\
Analogous results can be obtained in the   case where the manifold $M$ is compact,  extensively studied, e.g., in \cite{IkeWat}.  If $A_j$, $j=0,...,r$ are smooth vector fields, it is possible to construct a diffusion process $X=(X(t))$ solution of the stochastic PDE 
$$dX(t)=\sum_{j=1}^r A_j(X(t))\circ dB^j(t) +A_0(X(t))dt$$
where $\circ$ denotes the Stratonovich stochastic integral. 
The action of the Feller semigroup $V(t):C(M)\to C(M)$ given by $V(t)f(x)=\bE^x[f(X(t))]$ and the   generator extends the operator \eqref{Hoperator0} (see \cite{IkeWat,Hsu} for details).

However, we stress that  this technique does not directly  provide a core for the generator.
\end{remark}

 This section presents some sufficient conditions for the validity of the hypotesis (c) in Riemannian manifolds different form $\bR^d$.  
\begin{definition}\label{def22}\cite{shubin} Let $(M,g)$ a manifold of bounded geometry.  A differential operator of order $n$, $P : C^\infty(M) \to C^\infty(M)$, in local coordinates,
$$(P f)(x)= \sum_{|\alpha| \leq n}P_\alpha(x)  \partial^\alpha_xf$$
 is said to be 
{\bf $C^\infty$-bounded} if, for every $r_0 \in (0, I_{(M,g)})$ and every pair of multiindeces $\alpha, \beta $ there is a constant $C_{\alpha,\beta}\geq 0$ such that $|\partial ^{\beta}_xP_\alpha(x)| \leq  C_{\alpha, \beta}$ in every local Riemannian chart $(B_{r_0}^{(M,g)}, \exp^{-1}_{p,N})$ centered at every $p\in M$.
\end{definition}

\begin{remark} $\null$\\
{\bf(1)}  It is possible to prove \cite{shubin}
 that a $C^\infty$-bounded vector field $A$ on $M$ fulfills the following conditions   
 $$\| \|\nabla^{(g)k}  A\|_g\|_\infty \leq a_k\:, \quad k=0,1,\ldots\:.$$
 for some constants $a_k<+\infty$, $k=0,1,\ldots$.\\
{\bf(2)} It is possible to prove \cite{shubin} that if  a vector field $A$ on $M$  is  {\bf $C^\infty$-bounded},
then every differential operator given by the $p$-th power $A^p$ is  $C^\infty$-bounded. Obviously, linear combinations of $C^\infty$-bounded operators are $C^\infty$-bounded operators. Therefore the operator $L_0$ (\ref{Hoperator0}) is $C^\infty$-bounded if $(M,g)$ is of bounded geometry  and  the smooth vector fields $A_j$ are $C^\infty$-bounded for $j=0,\ldots,r$. \\
{\bf (3)}   
 Every $C^\infty$ vector field on a compact Riemannian manifold is automatically $C^\infty$-bounded.
Analogously, the operator $L_0$ (\ref{Hoperator0}) is $C^\infty$-bounded in the case the smooth Riemannian manifold $M$  is {\em compact} and the fields $\{ A_j\}_{j=0,\ldots ,r}$ are smooth.$\hfill \blacksquare$\end{remark}

From now on $\nabla^{(g)} \cdot A$ denotes the scalar field called {\bf covariant divergence} of $A$ completely defined in local coordinates around $p\in M$ as 
$$\nabla^{(g)}\cdot A := \sum_{j=1}^d (\nabla^{(g)}_j A)^j
 = \sum_{j=1}^d\left(  \partial_j A^j|_p+A^j\partial_j\log \sqrt{|g|}\right)\:.$$

Let us move on to state and prove the pivotal technical result of this section which we will use to prove that its closure $L= \overline{L_1}$ generates a Feller semigroup.  Everything relies upon the following technical result proved in the appendix and based on fundamental 
achievements by Shubin (Theorem 2.2  in \cite{shubin}), some of them  already established in \cite{Kor} where analytic semigroups in $L^p$-spaces 
 are in particular studied in manifolds of bounded geometry.
\begin{proposition}\label{teoL} Let $(M,g)$ be a smooth Riemannian manifold of bounded geometry and consider 
a uniformly elliptic 2nd order differential operator 
$L_0 :C^\infty(M) \to C^{\infty}(M)$ be of the form (\ref{Hoperator0}), where the $r\geq d$  real smooth  vector fields  $A_i$ 
are  $C^\infty$-bounded  and $A_0$ is defined
as 
\begin{equation}A_0 :=\frac{1}{2} \sum_{i=1}^r (\nabla^{(g)} \cdot A_i) A_i\:.\label{linkA}\end{equation}
Then,
\begin{itemize}
\item[(i)] $L := \overline{L_1}$ -- 
with $L_1 := L_0|_{D_k}$  and $D_k$ defined in (\ref{domainD}) -- is the  generator 
of a  Feller semigroup in $C_0(M)$ for every fixed $k=0,1, \ldots, \infty$. 

\item[(ii)] Both the generator  $L$ and the generated semigroup are  independent of $k$.
\end{itemize}
\end{proposition}
\begin{proof}
(i) What we have to prove is nothing but that the three hypotheses of theorem \ref{theoremFellergenerator}  are satisfied for $L_1 : D_k \to C_0(M)$.
Condition (a) has been established in proposition \ref{corollaryD}. Condition (b) immediatey arises from the form of $L_0$ and the
  ellipticity property it satisfies. Regarding (c), the pivotal result appears in the following lemma  proved in the appendix.

\begin{lemma}\label{propL}
With $(M,g)$ and $A_j$ ($j=0,\ldots,r$) and $L_0$ as in the hypothesis -- in particular $A_0$ 
 as in (\ref{linkA})-- for every $h \in C_c^\infty(M)$
and $\lambda >0$  there exists $f \in   C_0(M) \cap  C_b^\infty(M)$
fulfilling
\begin{equation} L_0f - \lambda f = h\:.\label{EQR2}\end{equation}
\end{lemma}
\begin{proof}
See the appendix.
\end{proof}

Now  observe that, due to lemma \ref{propL}, if $\lambda>0$ and  $h \in C_c^\infty(M)$, there 
is $f \in C_0(M)\cap C_b^\infty(M)$ (hence $f\in D_k$ for all $k=0,1,\ldots, \infty$)
such that $L_0f = \lambda f + h$. This fact can be rephrased to $(L_1 - \lambda I)f = h$. Since $C_c^\infty(M)$ is dense in $C_0(M)$  due to proposition \ref{corollaryD}, we have proved that $Ran(L_1-\lambda I)$ is dense in $C_0(M)$ for $\lambda >0$, demonstrating that also the hypothesis (c) in theorem \ref{theoremFellergenerator}  is satisfied. Let us finally prove (ii). This is consequence of the following general lemma.
\begin{lemma}\label{lemmaMN} Let $M: D(M) \to \mathcal{B}$ and $N : D(N) \to \mathcal{B}$ be two closed densely defined operators in the Banach space $\mathcal{B}$ which are generators of corresponding strongly continuous semigroups. If  $M \subset N$, then $M=N$. 
\end{lemma}
\begin{proof} See the appendix.
\end{proof}
\noindent The proof ends observing that 
$L_0|_{D_{k+1}} \subset L_{0}|_{D_{k}}$
so that $\overline{L_0|_{D_{k+1}}} \subset \overline{L_{0}|_{D_{k}}}$ and both operators are generators of strongly-continuous semigroups on $C_0(M)$. 
The case $D_\infty$ is encompassed since, e.g., $D_\infty \subset D_1$.
\end{proof}
We can finally prove the main result of this section, by  relaxing the requirement on the form of $A_0$.
\begin{theorem}\label{teoL2} Let $(M,g)$ be a smooth Riemannian manifold of bounded geometry and  consider   a 
uniformly elliptic 2nd order differential operator  $L_0 :C^\infty(M) \to C^{\infty}(M)$
 of the form (\ref{Hoperator0}), where $A_0$ and  the $r\geq d$  vector fields  $A_i$ are  real, smooth  and $C^\infty$-bounded.
 Then,
\begin{itemize}
\item[(i)] $L := \overline{L_1}$ -- 
with $L_1 := L_0|_{D_k}$  and $D_k$ defined in (\ref{domainD}) --  is the  generator 
of a  Feller semigroup in $C_0(M)$ for every fixed $k=0,1, \ldots, \infty$. 

\item[(ii)] Both the generator  $L$ and the generated semigroups are independent of $k$.
\end{itemize}
\end{theorem}
\begin{proof}
(ii) has the same proof as that of (ii) in proposition \ref{teoL}.
The proof of (i) is based on the following technical result.

\begin{lemma}\label{propL2} With $(M,g)$ and $A_j$ ($j=1,\ldots,r$) and $L_0$ as in the hypothesis assume that
\begin{equation}
A_0 := \frac{1}{2}\sum_{i=1}^r (\nabla^{(g)}\cdot A_i) A_i + B \label{A0new}\:,
\end{equation}
for  a real $C^\infty$-bounded vector field $B$.
If there exists $c>0$ independent of the used local chart around $x\in M$ such that 
\begin{equation}\label{dominance}
\sum_{a,b=1}^d B^a(x)B^b(x)\xi_a\xi_b \leq c\sum_{a,b=1}^d  \sum_{i=1}^r  A^a_i(x)A^b_i(x)\xi_a\xi_b \quad \mbox{for every 
$\xi_k \in \bR$ and every $x\in M$}
\end{equation}
then $L := \overline{L_1}$ -- with $L_1 := L_0|_{D_k}$ and $D_k$ defined in  (\ref{domainD})  is the  generator 
of a  Feller semigroup in $C_0(M)$.
\end{lemma}
\begin{proof}
See the appendix
\end{proof}
In view of lemma \ref{propL2}, to prove (i), it is sufficient to prove that (\ref{dominance}) is always satisfied however we choose the real smooth $C^\infty$-bounded vector field $B$. If we think of the numbers $\xi_k$ as the components of a form $\xi \in T^*_xM$, dividing both sides for $||\xi||^2_g\neq 0$, the inequality can be rephrased to, where $\langle \cdot, \cdot \rangle$ is the standard pairing on $T_xM\times T^*_xM$,
$$\frac{|\langle B(x), \xi(x) \rangle|^2}{||\xi||_g^2}\leq c \frac{\sum_{i=1}^r |\langle A_i(x), \xi(x) \rangle|^2}{||\xi||_g^2}\:.$$
The left-hand side above satisfies
$$\frac{|\langle B(x), \xi(x) \rangle|^2}{||\xi||_g^2} \leq \frac{||B(x)||^2_g||\xi||^2_g}{||\xi||^2_g} \leq || \|B\|_g||^2_\infty < +\infty$$
whereas the right-hand side fulfils 
$$\frac{\sum_{i=1}^r |\langle A_i(x), \xi(x) \rangle|^2}{||\xi||_g^2} \geq C \frac{||\xi||_g^2}{||\xi||_g^2} =C >0$$
just in view of the uniformly ellipticity condition. Choosing $c:=  || \|B\|_g||^2_\infty /C$, which is necessarily finite,  (\ref{dominance})  is satisfied.
\end{proof}
To conclude, we prove that we can modify $L_0$ by adding a  zero-order term in a certain class of continuous functions
preserving the results above. 
\begin{theorem}\label{teoL3} Let $(M,g)$ be a smooth Riemannian manifold of bounded geometry and  consider   a 
uniformly elliptic 2nd order differential operator  $L_{0c} :C^\infty(M) \to C(M)$
 of the form 
\begin{equation}L_{0c}:= L_0 + c\:, \label{L0u}\end{equation}
where $L_0$ is the operator defined in theorem \ref{teoL2} and  $c\in C^0_b(M)$ being bounded and continuous, defines
 a multiplicative operator 
$c\in \mathscr{L}(C_0(M))$.
 Then,
\begin{itemize}
\item[(i)] Assuming additionally that $c(x) \leq 0$ for all $x\in M$ we obtain that $L := \overline{L_{1c}}$   -- 
with $L_{1c} := L_{0c}|_{D_k}$ and $D_k$ defined in (\ref{domainD}) --  is the  generator 
of a  Feller semigroup in $C_0(M)$ for every fixed $k=0,1, \ldots, \infty$. 

\item[(ii)] Under condition $c(x) \leq 0$ for all $x\in M$ both the generator  $L$ and the generated semigroups are independent of $k$.

\item[(iii)] 
If condition $c(x) \leq 0$ for all $x\in M$ does not hold, then $L$ as in (i) is still the generator of a strongly continuous semigroup in $C_0(M)$ for every fixed $k=0,1, \ldots, \infty$, and  (ii) is still valid.
\end{itemize}

\end{theorem}
\begin{proof}
(i) Let us start the proof by establishing that the multiplicative operator $-c$ is {\em accretive}  (\cite{RS2} Definition on p. 240). In fact, if 
$f\in C_0(M)$, let $p\in M$ be such that $|f(p)|= \sup_{x\in M} |f(x)|$.
Let us construct a {\em normalized functional} $\lambda \in C_0(M)'$ {\em tangent} to $f\in C_0(M)$ as
$$\lambda(h) :=  \overline{f(p)} h(p) \:, \quad h \in C_0(M)\:,$$
It holds trivially $||\lambda||=||f||$ and $\lambda(f) = ||f||^2$ so that $\lambda$ is normalized and tangent to $f$, and also $\lambda((-c)f)\geq 0$
 (notice that $c\leq 0$), so that $-c$ is accretive. At this juncture we can apply the lemma on p. 244 of \cite{RS2}
with $0\leq a<1/2$,  $b:= \sup_M |c|$,
$A:= -\overline{L_0|_{D_k}}$, and\footnote{Notice that in \cite{RS2} semigroups are represented as $e^{-tA}$ whereas for us they are 
represented as $e^{tL}$ this explains the sign minus in front of the operators.} $B:= -c \in \mathscr{L}(C_0(M))$.
 Since $\overline{L_0|_{D_k}}$
generates a Feller semigroup which is a contraction semigroup by definition, we conclude from the above lemma
 that $\overline{L_0|_{D_k}} + c$ is the generator of a contraction semigroup.
Since $c \in  \mathscr{L}(C_0(M))$, we also have  $\overline{L_0|_{D_k}} + c = \overline{(L_0+c)|_{D_k}} = \overline{L_{0c}|_{D_k}}$.
According to definition \ref{semFellerdef} of Feller semigroup, the proof of (i)  ends by proving that the generated semigroup of contractions is made of positive operators.
This fact immediately arises from the Trotter product formula
$$e^{-t\overline{A+B}}f = \lim_{n\to +\infty} \left(e^{-tA/n} e^{-tB/n}\right)^n f\:,$$
i.e.,  Theorem X.51 in \cite{RS2}, with $A=-\overline{L_{0c}|_{D_k}}$ and 
$B= -c$, which is valid because $A+B$ generates a contraction semigroup as established above.
Now observe  that $e^{-tA/n}$ is positive, since it is an element of a Feller semigroup, and  $e^{-tB/n}$ is positive as well just because, 
by direct inspection, it is nothing but the multiplicative operator with a positive function
$e^{tc(x)}$. Since  the limit in the Trotter formula  here  is computed with respect to the norm $||\cdot||_\infty$, we find 
 $e^{-t\overline{A+B}}f \geq 0$ if $f\geq 0$, so that the semigroup generated by $L$ is made of positive elements and the proof
 of (i) ends.
 
 The proof of (ii) is identical 
to that of (ii) in theorem \ref{teoL2}.

To prove (iii) 
 it is sufficient to write $c(x)=\tilde c(x) +\sup_x c(x)$ with $\tilde c=c-\sup_x c(x)$ and apply items (i) and (ii) to $L_0+\tilde c$, noting that the added constant $\sup_x c(x)$ does not affect domains and closures.  The resulting semigroup $V_c(t)$ has the form $V_c(t)=e^{t\sup_x c(x)}V_{\tilde c}(t)$, where $V_{\tilde c}(t)$ is the Feller semigroup generated by $\overline{L_{0\tilde c}|_{D_k}}$. 
\end{proof}

\subsection{Chernoff functions for the Feller semigroup }\label{sez3.3}
In this section we discuss how the Feller semigroup $V(t)$ generated by $L$ can be obtained by a suitable Chernoff function $S$ again constructed out of the vector fields $A_j$. 

In the following we shall assume that the smooth Riemannian manifold  $(M,g)$ is of bounded geometry. In particular this implies that $(M,g)$ is geodesically complete (see  definition \ref{defboundedgeom} and lemma \ref{lemmaC}).

\begin{theorem}\label{teo3.1}
Let $(M,g)$ be a smooth Riemannian manifold of bounded geometry and  consider   a 
uniformly elliptic 2nd order differential operator  $L_0 :C^\infty(M) \to C^{\infty}(M)$
 of the form (\ref{Hoperator0}), where $A_0$ and  the $r\geq d$  vector fields  $A_i$ are  real, smooth  and $C^\infty$-bounded.  Let $c \in C_b^0(M)$  and let $L_{0c}:=L_0+c $ and 
$L := \overline{L_{1c}}$ -- with $L_{1c} := L_{0c}|_{D_k}$ and $D_k$ defined in (\ref{domainD}) for $k=0,1,\ldots, \infty$. \\
For any $x\in M$, $t\geq 0$ and $f\in C_0(M)$ let us define 
\begin{equation}\label{S(t)}(S(t)f)(x)=\frac{1}{4r}\sum_{j =1}^r\left(f\left(\gamma_{x,A_j}(\sqrt{2rt})\right)+
f\left(\gamma_{x,-A_j}(\sqrt{2rt})\right)\right)+\frac{1}{2}f(\gamma_{x,A_0}(2t)) + tc(x)f(x).
\end{equation}
 where  $\gamma_{x,A_j}:\bR^+\to M$ is the integral  curve of the vector field $A_j$ starting at time $t=0$ at the point $x\in M$, namely the solution of the initial value problem 
\begin{equation}\label{CPAk}
\left\{ \begin{array}{ll}
\frac{d}{dt} \gamma_{x,A_j} (t)=A_j(\gamma_{x,A_j}(t)), \\
\gamma_{x,A_j} (0)=x.
\end{array} \right.
\end{equation}
Then the following holds.
\begin{enumerate}
    \item For all $t\geq 0$ $S(t)(C_0(M))\subset C_0(M)$.
    \item  If $(V(t))_{t\geq 0}$ is the strongly continuous
     semigroup on $C_0(M)$ generated by $L$ (according to theorems \ref{teoL2} and \ref{teoL3})
    then for any $f\in C_0(M)$ and $T>0$ the following holds
\begin{equation}\label{convergenceformula}\lim_{n\to \infty}\sup _{t\in [0,T]}\|S(t/n)^nf-V(t)f\|=0\:.\end{equation} 
\end{enumerate} 
\end{theorem}
\begin{proof}
We remark that the right hand side of \eqref{S(t)} is well defined for all $t\geq 0$ since by lemma 
\eqref{lemma-complete-solutions} the maximal solution of the Cauchy problem \eqref{CPAk} 
is defined for all $t\geq 0$, the manifold $(M,g)$ being geodesically complete by the assumption of bounded geometry. 
Let us first assume $c=0$.
\begin{enumerate}

 \item 
 The continuity of the functions   $x\longmapsto f(\gamma_{x,A_0}(2t))$ and $x\longmapsto f(\gamma_{x,A_j}(\sqrt{2rt}))$, $j=1,\dots , r$,  follows from the continuity of the maps $x\longmapsto \gamma_{x,A_j}(\tau)$ for all $j=0,\dots , r$ and $\tau \in \bR^+$.  Moreover, if $f\in C_0(M)$, then for any $x\in M$, $\tau\in \bR^+$ and $k=0,...,r$, the map   $x\longmapsto f(\gamma_{x,A_j}(\tau))$ belongs to $C_0(M)$
proving 1 in th thesis. Indeed, given $\varepsilon >0$  there exists a compact set $K_\varepsilon$ such that $|f(y)|<\varepsilon $ for $y\in K_\varepsilon^c$. Set $\sup _{x\in M}\|A_j(x)\|:=c_j<\infty$ and consider the set $K_{\varepsilon, \tau}$  defined as the closure of the set of points $y\in M$ whose distance from $K_\varepsilon$ is less then $c_j\tau$:
\begin{equation}\label{ket}K_{\varepsilon, \tau}:=\overline{\{y\in M\:|\: d(y, K_\varepsilon)\leq c_j\tau \}},\end{equation}
where $d(y, K_\varepsilon):=\inf _{x\in K_\varepsilon} d(y,x)$. Since $K_\varepsilon$ is compact, it is bounded,  namely it is
contained  in some closed geodesical ball of finite radius $R$ centered on some $x_0\in M$. Therefore, the closed 
 set $K_{\varepsilon, \tau}$ is  bounded as well since it is enclosed in a closed  ball of radius $R+c_j\tau$ centered on $x_0$
and it is therefore  compact by the Hopf-Rinow theorem because  $(M,g)$ is complete. If $x \in K_{\varepsilon, \tau}^c$ then $\gamma_x(\tau) \in K_\varepsilon ^c$, 
hence $|f(\gamma_{x,A_j}(\tau))|<\varepsilon$. Indeed if this was not true, i.e. if $\gamma_{x,A_j}(\tau)\in K_\varepsilon$, then:
$$d(x, K_\varepsilon)\leq d(x,\gamma_{x,A_j}(\tau))\leq \int_0^\tau \|\dot\gamma_{x,A_j}(s)\|ds= \int_0^\tau \|A_j(\gamma_{x,A_j}(s))\|ds<c_j\tau.$$\\

\item 
\begin{enumerate}
    \item 
First of all we prove that if $f\in C_0(M)$ then  $\sup_{x\in M}|(S(t)f)(x)|\leq \sup_{x\in M}|f(x)|.$\\
Indeed, for all $x\in M$ we use the fact that function $f$ is bounded and obtain
$$
|(S(t)f)(x)|\leq\frac{1}{4r}\sum_{k =1}^r\left(\left|f\left(\gamma_{x,A_j}(\sqrt{2rt})\right)\right|+\left|f\left(\gamma_{x,-A_j}(\sqrt{2rt})\right)\right|\right)+\frac{1}{2}|f(\gamma_{x,A_0}(2t))|
$$
$$
\leq\frac{1}{4r}\sum_{k =1}^r\left(2\sup_{z\in M}|f(z)|\right) + \frac{1}{2}\sup_{z\in M}|f(z)|=\sup_{z\in M}|f(z)|.
$$\\

\item The mapping $\bR^+\ni t\longmapsto S(t)f\in C_0(M)$ is continuous.\\
It is sufficient to show that for any $k=0,\dots r$ the map $\bR^+ \ni \tau \longmapsto S_j(\tau )f\in C_0(M)$ given by $S_j(\tau )f(x):=f(\gamma_{x,A_j}(\tau))$ is continuous in the $\sup$-norm.\\
  Let $\tau _0\in \bR^+$ and fix $\varepsilon >0$. Since $f\in C_0(M)$, there exists a compact set $K_\varepsilon $ such that $|f(y)| <\varepsilon/2$
 for $y\in K_\varepsilon ^c$. If $c_j:=\sup _{x\in M}\|A_j(x)\|$ and considering the compact set 
$K_{\varepsilon, \tau}$ defined in \eqref{ket} with $\tau =\tau _0+1$, we have that if $t\in [0, \tau_0+1]$ then
 $\gamma_{x,A_j}(t)\in K^c_\varepsilon $ for any $x\in K^c_{\varepsilon, \tau_0+1}$, hence
$$|f(\gamma_{x,A_j}(\tau))- f(\gamma_{x,A_j}(\tau_0))|<\varepsilon, \qquad \forall x\in K^c_{\varepsilon, \tau_0+1}.$$
 If $x\in K_{\varepsilon, \tau_0+1}$, then for $t\in [0, \tau _0+1]$ we have $\gamma_{x, A_j}(t)\in  K'_{\varepsilon, \tau_0+1}$, 
where $K'_{\varepsilon, \tau_0+1}$ is the compact set defined as
  $$K'_{\varepsilon, \tau_0+1}=\overline{\{y\in M \:|\:d(y, K_{\varepsilon, \tau_0+1})\leq c_j(\tau_0+1) \}}.
$$
Since $f$ is continuous on $M$, it is uniformly continuous on the compact set $K'_{\varepsilon, \tau_0+1}$ and for any $\varepsilon >0 $ there exists a $\delta >0$ such that $|f(x)-f(y)|<\varepsilon$ for $x,y\in K'_{\varepsilon, \tau_0+1}$ such that $ |x-y| <\delta$. If $x\in K_{\varepsilon, \tau_0+1}$ and $|\tau-\tau_0|<\min\{1,\delta /c_j\}$, then $\gamma_x(\tau), \gamma_x(\tau_0)\in K'_{\varepsilon, \tau_0+1}$ and $|\gamma_x(\tau)-\gamma_x(\tau_0)|<\delta$, hence:
$$|f(\gamma_{x,A_j}(\tau))- f(\gamma_{x,A_j}(\tau_0))|<\varepsilon, \qquad \forall x\in K_{\varepsilon, \tau_0+1}.$$

\item  If  $\varphi$ belongs to the core $D_k$ of $L$ with $k\geq 3$  we have $$S(t)\varphi=\varphi+tL_1\varphi+o(t) \quad \mbox{as $\bR^+ \ni t\to 0$ 
in the uniform norm}$$
 -- where $D_k$ is defined in \eqref{domainD} and $L_1 := L_0|_{D_k}$ with $L_0$ defined in (\ref{Hoperator0}).\\

 For fixed $x\in M$ and $k\in \{1,\dots ,r\}$ let us consider the map $t\mapsto \varphi (\gamma_{x,A_j}(t))$ which is smooth by the stated assumptions on $\varphi \in D_k$ and $A_j$. By Taylor expansion we have for $t\downarrow 0$:
\begin{align}
\varphi (\gamma_{x,A_j}(t))&=\varphi (\gamma_{x,A_j}(0))+t\frac{d}{dt}|_{t=0}\varphi (\gamma_{x,A_j}(t)) +\frac{t^2}{2}\frac{d^2}{dt^2}|_{t=0}\varphi (\gamma_{x,A_j}(t))+R_j(x,t)\\
&=\varphi (x)+t\left(A_j\varphi\right)(x)+\frac{t^2}{2}\left(A_jA_j\varphi\right) (x)+R_j(x,t), \label{Taylor1}\\
\end{align}
where $$R_j(x,t)=\frac{t^3}{3!}\left(A_jA_jA_j\varphi\right)(u),$$ with $u=\gamma _{x,A_j}(\xi)$, $\xi\in [0, t]$. Analogously for $j=0$ we have:
\begin{align*}
\varphi (\gamma_{x,A_0}(t))&=\varphi (\gamma_{x,A_0}(0))+t\frac{d}{dt}|_{t=0}\varphi (\gamma_{x,A_0}(t)) +R_0(x,t)\\
&=\varphi (x)+t\left(A_0\varphi\right)(x)+R_0(x,t), \label{Taylor2}
\end{align*}
with $R_0(x,t)=\frac{t^2}{2!}\left(A_0A_0\varphi\right)(u),$ with $u=\gamma _{x,A_0}(\xi)$, $\xi\in [0, t]$. Hence
\begin{align*}
S(t)\varphi (x)&= \frac{1}{4r}\sum_{j =1}^r\left(\varphi\left(\gamma_{x,A_j}(\sqrt{2rt})\right)+
\varphi\left(\gamma_{x,-A_j}(\sqrt{2rt})\right)\right)+\frac{1}{2}\varphi(\gamma_{x,A_0}(2t))\\
&=\varphi(x)+\frac{1}{4r}\sum_{j =1}^r 2rt \left(A_jA_j\varphi\right) (x) +t\left(A_0\varphi\right)(x)+ t^{3/2}\tilde R(t,x)\\
&= \varphi(x) +tL_1\varphi(x)+t^{3/2}\tilde R(t,x)
\end{align*}
where 
$$\tilde R(t,x)=\sqrt t (A_0A_0 \varphi)(u_0)+\frac{\sqrt{2r}}{12}\sum_{j=1}^r\left(\left(A_jA_jA_j\varphi\right)(u_j)+\left(A_jA_jA_j\varphi\right)(u'_j)\right),$$
for suitable $u_0, u_j, u'_j\in M$, $j=1,\ldots, r$. The proof concludes by proving that  $\sup_{t\in [0,1], x\in M}|\tilde R(x,t)|<\infty$. This fact arises from the bounds
$$\|(A_0A_0 \varphi)\|_\infty\:,\quad  \|(A_jA_jA_j\varphi)\|_\infty\:,  j=1, \ldots, r, 
$$
 due to the very definition (\ref{domainD}) of $D_k$ as well as on the assumption that the vector fields $\{A_j\}_{j=0, \ldots r}$ are $C^\infty$-bounded and $\varphi \in D_k$ with $k\geq 3$.
 \end{enumerate}
 This concludes the proof of (2)  since the conditions (1)-(4) in theorem \ref{FormulaChernova}
assuring the validity of (2) are valid
 in view of the results above ((3) is trivially true).
\end{enumerate} 
 The case $c\neq 0$ has now an easy  proof. Let  $S_0$ denote the Chernoff function of $L$ with $c=0$ 
and let $S$
denote  the analog for the case  $c\neq 0$.  If $f\in C_0(M)$ then  $S(t)f = S_0(t)f + tcf \in C_0(M)$ because $S_0(t)f \in C_0(M)$, $f\in C_0(M)$ and 
$c$ is continuous and
 bounded. Hence (1) is true. Regarding (2), the estimate 
 $\|S(t)f\|=\|S_0(t)f+tcf\|\leq \|S_0(t)\|||f||+t\|c\|\|f\|=(1+t\sup_{x\in M}|c(x)|)||f||\leq e^{t\|c\|}||f||$
proves that condition (1)  in theorem \ref{FormulaChernova} is valid. Requirement (2) is valid because $S= S(t)$ is the sum of 
two continuous $\mathscr{L}(C_0(M))$-valued functions of $t$. (3) is trivially true. Condition (4) is satisfied because if $\varphi \in D_k$
with $k\geq 3$, exploiting condition (c) in (2) above, and where $L_1$ is referred to the case $c=0$,
 $$S(t)\varphi=S_0(t)\varphi+tc\varphi=\varphi+tL_1\varphi+o(t)+tc\varphi=\varphi+t(L_{1} + c)\varphi+o(t)
= \varphi+tL_{1c}\varphi+o(t)\:. $$
Hence theorem \ref{FormulaChernova} implies that (2) is valid.
\end{proof}
\begin{theorem} Under assumptions of theorem \ref{teo3.1}, the following facts hold.

\begin{itemize}
\item[(1)] For the operator  $L$  defined in theorem \ref{teoL3} and $S(t)$  defined in (\ref{S(t)}), we have that the classical solution\footnote{In sense of Proposition \ref{ACPsol}.} $u$ of the Cauchy problem 
 \begin{equation*}
 \left\{ \begin{array}{l}
 \frac{\partial }{\partial t}u(t,x)=Lu(t,x)\\
 u(0,x)=u_0(x)
 \end{array}
 \right.
 \end{equation*}
  is given for $u_0\in D(L)$ by \begin{equation}\label{C-M-1}u(t,x)=\lim_{n\to\infty}(S(t/n)^{n }u_0)(x).\end{equation}
\item[(2)] In the case  $A_0=0$ and $c=0$, then an alternative equivalent form for the operator $S(t):C_0(M)\to C_0(M)$, $t\geq 0$,  is:
\begin{equation}\label{S(t)-vers-2}
    (S(t)f)(x)=\frac{1}{2r}\sum_{j =1}^r\left(f\left(\gamma_{x,A_j}(\sqrt{2rt})\right)+f\left(\gamma_{x,-A_j}(\sqrt{2rt})\right)\right), \qquad f\in C_0(M)
\end{equation}

\end{itemize}
  \end{theorem}
  \begin{proof}
  Result (1) immediately arises from  \eqref{convergenceformula}, which is valid  for all $f\in C_0(M)$, for  all $x\in M$, and all $t\geq 0$.
  (2) It can be proved with a proof strictly analogous to that of the corresponding statement in the theorem  \ref{teo3.1}.
  \end{proof}

\section{A probabilistic interpretation of Chernoff construction}\label{SecProbabilisticInterpretation}

The convergence result stated by Chernoff construction can be equivalently formulated (see \cite{EN1} Th 5.2 Ch. III) in the following way for all  $t\geq 0$ and $f\in C_0(M)$:
$$V(t)f=\lim_{n\to\infty}(S(1/n)^{\lfloor nt\rfloor }f)\:.$$
Assuming that the function $c=0$, the latter formula admits a probabilistic interpretation in terms of the limit of expectations with respect to a sequence of random walks on the manifold $M$. Actually, in the following sections we shall set $c=0$ and  provide three different constructions.
\subsection{A jump process on $M$}
Let  $\{X_n(t)\}_{n\geq 1} $ be a sequence of jump processes on $M$ defined as
\begin{equation}\label{Xnjump}
\begin{cases}
X_n(0)\equiv x ,\\ X_n(t):=X_n(\lfloor  nt\rfloor/n)=Y_n(\lfloor  nt\rfloor) \quad t>0,
\end{cases}\end{equation}
the jump chain $\{Y_n(m)\}_{m\geq 1}$ is a  Markov chain with transition  probabilities  given (for each Borel set $B\subset M$) by 
\begin{multline} \bP(Y_n(m)\in B|Y_n(m-1)=y)= \\ =\frac{1}{4r}\sum_{j=1}^r\left(\delta_{\gamma_{y, A_j}\left(\sqrt{2r/n}\right)}\left(B\right) + \delta_{\gamma_{y, -A_j}\left(\sqrt{2r/n}\right)}\left(B\right)\right)+\frac{1}{2}\delta_{\gamma_{y, A_0}(2/n)}\left(B\right), \quad B\in \cB(M).\end{multline}
Actually $(X_n(t))_{t\geq 0}$
is a random walk on $M$ with  steps given by integral curves of the vector fields $A_k$, $k =0,\dots r$. 
Now equation \eqref{C-M-1} can be written in the following form:
\begin{equation}u(t,x)=\lim_{n\to\infty}(S(1/n)^{\lfloor nt\rfloor }u_0)(x)
=
\lim_{n\to\infty}\bE[u_0(X_n(t))]\label{limit1}\end{equation}
Actually, the sequence of jump processes $\{X_n\}$ converges weakly to  the  diffusion process   $(X(t))_{t\in \bR^+}$ on  $M$ associated to the Feller semigroup $V(t)$, as we are going to show.

Let $D_{M}[0,+\infty)$  denote the space of c\'adl\'ag $M-$valued functions over the interval $[0,+\infty)$, i.e the functions which are right-continuous and admit left hand limits. It is possible to define  a distance function (i.e. metric) on $D_{M}[0,+\infty)$ under which it becomes a separable metric space. The topology induced by the metric is called {\em Skorohod topology} \cite{Bil,EthKur}.
In the following, with the symbol $\cS_M$ we shall denote the corresponding Borel $\sigma-$algebra on $D_{M}[0,+\infty)$. In fact $\cS_M$ coincides with the $\sigma-$ algebra generated by the projection maps $\pi_t:D_{M}[0,+\infty)\to M$
\begin{equation}\label{sigmaDM}
    \cS_ M=\sigma (\pi_t, t\geq 0)
\end{equation}
where
\begin{equation}\label{pit}\pi_t(\gamma):= \gamma (t), \qquad \gamma \in D_{M}[0,+\infty).\end{equation}
 
As a consequence, a stochastic process $X=(\Omega, \cF , \cF_t, (X(t)), \bP)$ with  trajectories in  $D_{M}[0,+\infty)$ can be looked at as a $D_{M}[0,+\infty)-$valued random variable, i.e. as a map $X:\Omega \to D_{M}[0,+\infty)$ defined as:
$$X(\omega):=\gamma _\omega, \qquad \gamma_\omega (t):=X(t)(\omega), \qquad t\in [0, +\infty), \, \omega \in \Omega. $$
The measurability of the map $X$ from $(\Omega, \cF)$ to $(D_{M}[0,+\infty), \cS_M)$ follows from \eqref{sigmaDM}.
We shall denote with $\mu_X$ the probability measure on $\cS_M$ obtained as the pushforward of $\bP$ under $X$, defined for any Borel set $I\in \cS_M$ as $\mu_X(I)=\bP(X(\omega )\in I)$. 

Considered the sequence of jump processes $(X_n)$ defined on a probability space $(\Omega, \cF, \bP)$ by \eqref{Xnjump}, let $\mu_{X_n}$ be the corresponding distribution on $(D_{M}[0,+\infty), \cS_M)$. Further, let $\mu_X$ be the analogous distribution corresponding to the Feller process $X$.
\begin{theorem}\label{thconvD}Under the assumptions of theorem \ref{teo3.1},
the sequence of processes $X_n$ converges weakly in $D_M[0,+\infty)$ and  its weak limit is the Feller process $X$.
\end{theorem}
\begin{proof} The proof is a direct application of  \eqref{limit1} and of theorem 2.6 Ch 4 of \cite{EthKur}, see also theorem 19.25 in \cite{Kal}.\end{proof}
\subsection{A piecewise geodesic random walk }
For any $T>0$, let us consider the space  $C_{M}[0,T]$ of continuous functions $\gamma :[0, T] \to M$ endowed with the topology of uniform convergence. The corresponding Borel $\sigma $-algebra $\cB_M$ is generated by the coordinate projections $\pi_t$, $t\in [0,T] $ defined as above (see Eq. \eqref{pit}) \cite{Bal}.

The stochastic  process $X$ associated to the Feller semigroup $V(t)$ is a diffusion process, hence it has continuous trajectories. 

Let us consider the sequence of processes $(Z_n)_n$ with sample paths in $C_{M}[0,T)$, obtained by continuous interpolation of the paths of $(X_n)_n$ by means of geodesic arcs.
More precisely, the process $(Z_n(t))_{t\geq 0}$ is defined as
\begin{equation}
    \begin{cases}
    Z_n (0)\equiv x, \\
    Z_n(m/n)\equiv X_n(m/n), \quad m\in \bN,\\
     \quad Z_n(t)=\gamma_{X_n(m/n),X_n((m+1)/n)}(t-m/n), \: t\in [m/n, (m+1)/n]
    \end{cases}
\end{equation}, 
where $\gamma_{x,y}(t)$ denotes an arbitrary shortest geodesics in $M$ such that $\gamma_{x,y}(0)=x$ and $\gamma_{x,y}(1/n)=y$.

Let us denote with $\mu_n$, resp. $\mu$, the Borel measure over the space $C_{M}[0,T]$ induced by the process $Z_n$, resp. $X$.  The following holds.
\begin{theorem}\label{teoconvC}
Under the assumptions above, $Z_n$ converges to $X$ on $C_{M}[0,T]$. 
\end{theorem}
In other words,  theorem \ref{teoconvC} states that  the sequence of measures $\{\mu_n\}$ converges weakly to $\mu$.
Before proving theorem \ref{teoconvC} we recall some preliminary results.

\begin{definition}
Let $(M,d)$ be a separable metric space. The {\em modulus of continuity } of a function $\gamma:[0,T]\to M$ is defined for any  $\delta>0$ as: 
$$w(\gamma, \delta):=\sup\{d(\gamma (t),\gamma(s)), s,t\in [0,T], |t-s|<\delta\}.$$
\end{definition}

\begin{lemma}\label{lemma4}
  let $\nu_n$ be a sequence of probability measures on  $D_M[0,T]$ converging weakly to a finite measure $\nu$ which is concentrated on $C_M[0,T]$. Then for any $\varepsilon >0$ 
  \begin{equation}
      \lim_{\delta \downarrow 0}\limsup_{n\in \bN}\nu_n(\{\gamma\in D_M[0,T]\colon w(\gamma , \delta )>\varepsilon\})=0
  \end{equation}
\end{lemma}
For a proof see \cite{SWW2007}.
\begin{proof}[Proof of theorem \ref{teoconvC}]
 Let us consider the trajectories $\gamma_\omega$ of the process $Z_n$, defined as $\gamma_\omega (t):=Z_n(t)(\omega)$. Fix $\delta >0$ and take $n$ sufficiently large in such a way that $1/n<\delta$. Consider $s,t\in [0,T]$, $s<t$, $|t-s|<\delta$. We will have   $s\in [m/n,(m+1)/n]$ and $t\in [m'/n,(m'+1)/n]$, with $m\leq m'$.   We have:
\begin{align*}
    & d(\gamma_\omega (s), \gamma _\omega (t))\\
    &\leq
    d\left(\gamma_\omega(s),\gamma_\omega((m+1)/n)\right)+d\left(\gamma_\omega((m+1)/n),\gamma_\omega(m'/n)\right)+ d\left(\gamma_\omega(m'/n),\gamma_\omega(t)\right)\\
    &\leq d\left(\gamma_\omega(m/n),\gamma_\omega((m+1)/n)\right)+d\left(\gamma_\omega((m+1)/n),\gamma_\omega(m'/n)\right)+ d\left(\gamma_\omega(m'/n),\gamma_\omega((m'+1)/n)\right)\\
    &\leq 3\max\{ d\left(\gamma_\omega(m/n),\gamma_\omega(m'/n)\right), |m/n-m'/n|<\delta\}
\end{align*}

We can then estimate the probability that the modulus of continuity of the trajectories of $Z_n$ exceeds a given $\varepsilon >0$ as
\begin{eqnarray*}
&&\mu_n \left(\{\gamma\in C_M[0,T]\colon w(\gamma ,\delta)>\varepsilon\}\right)\\
&&\leq \mu_n \left(\{\gamma\in C_M[0,T]\colon \max_m\{d(\gamma (m/n),\gamma (m+1)/n))\}>\varepsilon/3\}\right)\\
& &=\mu_{X_n}\left(\{\gamma\in D_M[0,T]\colon
w(\gamma ,\delta)>\varepsilon /3\}\right)
\end{eqnarray*}
By theorem \ref{thconvD} and lemma \ref{lemma4}, we get for any $\varepsilon>0$
$$\lim_{\delta\downarrow 0}\limsup_n \mu_n \left(\{\gamma\in C_M[0,T]\colon w(\gamma ,\delta)>\varepsilon\}\right) =0$$
Since $Z_n(0)=x$ for any $n$, the sequence of probability measures $\{\mu_n\}$ is tight \cite{Bil} and the measure $\mu$, i.e. the law of $X$ is the only possible limit point.

\subsection{A different interpolation scheme}
Let us consider the sequence of processes $(\tilde Z_n)_n$ with sample paths in $C_{M}[0,T)$, obtained by continuous interpolation of the paths of $(X_n)_n$ by means of integral curves of the vector fields $A_k$, $k=0,...,r$.
More precisely, introduced a sequence of i.i.d. discrete random variables $\xi_j$ with distribution 
$$\bP(\xi_j =0)=1/2, \qquad \bP(\xi_j =k)=\frac{1}{2r}, \quad k=1,\dots, 2r,$$
and the map $\tau:\{0,\dots,2r \}\times [0,1]\to \bR$ defined by
$$\tau(k, t)=\begin{cases}
2t & k=0\\
\sqrt{2rt} & k=1,\dots 2r
\end{cases}$$
the process $(\tilde Z_n(t))_{t\in \bR^+}$ can be defined as 
\begin{equation}\label{deftildeZ-n}
\begin{cases}
\tilde Z_n (0)\equiv x,\\
\tilde Z_n(t)= \gamma_{\tilde Z_n(m/n), (-1)^{\xi_m}A_{\xi_m/2}}(\tau (\xi_m, t-m/n)) \qquad t\in [m/n,(m+1)n], 
\end{cases} 
\end{equation}
where for a$x\in M$ and a smooth vector field $A$ on $M$, $\gamma _{x,A}$ denotes the maximal solution of the Cauchy problem \eqref{CP-M}.
In particular the following holds:
\begin{equation*}
\quad \tilde Z_n (m/n)=X_n(m/n), \;\: m\in \bN.
\end{equation*}
\end{proof}
Analogously to the case of geodesic interpolation studied in the previous section, it is possible to prove the weak convergence of $\tilde Z_n $ to $X$ on $C_M[0,T]$. Let $\tilde \mu _n$ (resp. $\mu$) be the Borel probability measure on $C_M[0,T]$ induced by the process $\tilde Z_n$ (resp. $X$).

\subsubsection{A technical interlude}
In this subsection we introduce some results that will be applied to the proof of theorem \ref{teo-conv-lastZn}.

In this section, if $t=
\sum_{i=1}^dt^ie_i$  and $s=
\sum_{i=1}^ds^ie_i$,
where $(e_j)_{j=1,\ldots,d}$ is the standard orthonormal basis of $\bR^d$,
  $$\|t\|:= \sqrt{\sum_{i=1}^d (t^i)^2}\quad\mbox{and}\quad \langle t, s\rangle := \sum_{i=1}^d t^is^i$$  respectively denote the 
standard Euclidean norm and the standard inner product in $\bR^d$.  Furthermore, $d_{\bR^d}(p,q) := \|p-q\|\in [0,+\infty)$ denotes the usual Euclidean distance of $p,q \in \bR^d$.
\\
Let us start by considering the case where $M=\bR^d$.
\begin{proposition}\label{TRd}
Let $A\colon\bR^d\to\bR^d$ 
be a smooth vector field such that, for some $M_1,M_2 \in (0,+\infty)$,
\begin{enumerate}
    \item $\|A(x)\|\leq M_1$ if $x\in \bR^d$
    \item the components $A^i\colon\bR^d\to\bR$  satisfy 
    $\|\nabla A^i(x)\|\leq M_2$  for all $i=1,..,d$ if $x\in \bR^d$.
\end{enumerate}
Consider the unique maximal and complete (for 1) smooth  solution $\gamma\colon \bR \to \mathbb{R}^d$ of the Cauchy problem 
\begin{equation}\label{CPZn}
\left\{\begin{array}{l}
     \dot \gamma (t)=A(\gamma(t))\\
     \gamma (0)=\gamma_0
\end{array}\right.
\end{equation}
for every $\gamma_0 \in \bR^d$ and define $d_{\gamma_0}:[0,+\infty) \to \bR$  as $$d_{\gamma_0}(t):= d_{\bR^d}(\gamma(0), \gamma(t)).$$
Then there exists  $T>0$ independent of $\gamma_0$ such that  the function $d_{\gamma_0}$ is   non-decreasing in $[0,T]$. 
Even more, $d_{\gamma_0}$ is strictly increasing in $[0,T]$  if $A(\gamma_0)\neq 0$. 
\end{proposition}
\begin{proof}
First of all let us remark that if $A(\gamma(0))=0$ then $d_{\gamma_0}(t)=d_{\bR^d}(\gamma(0), \gamma(t)) =0$ 
and the result holds trivially for any $T>0$. 
Let us therefore  restrict ourselves to the case  $A(\gamma(0))\neq 0$ where, by the local uniqueness 
of the solutions of the Cauchy problem  \eqref{CPZn}, 
we have that $A(\gamma(t))\neq 0$ for all $t\neq 0$.
 Let $f_{\gamma_0}\colon[0,+\infty) \to \mathbb{R}$ be the  smooth map $f_{\gamma_0}(t)=d_{\gamma_0}(t)^2=\|\gamma (t)-\gamma (0)\|^2$. 
To prove the   thesis it is enough to demonstrate  that
\begin{equation} \mbox{\em if $A(\gamma_0)\neq 0$, then  there exists $T>0$ independent of $\gamma_0$ such that 
 $f_{\gamma_0}'(t)>0$ for all $t\in (0,T]$.} \label{newstat} \end{equation}
To prove  (\ref{newstat}), we start by notincing that
 linearity and symmetry of the inner product in $\mathbb{R}^d$ and
the trivial identity arising from (\ref{CPZn})
\begin{equation}\gamma(t)-\gamma(u) = \int_u^t A(\gamma(s))ds\label{easyEq}\end{equation}
 yield
\begin{equation}
 f_{\gamma_0}(t)=\left<\int_0^t A(\gamma (s))ds,\int_0^t A(\gamma (u))du\right>=
      2\int_0^t\int_0^s \left<A(\gamma (s)),A(\gamma (u))\right>duds.\nonumber
\end{equation}
The derivative $f_{\gamma_0}'(t)$ appearing in  (\ref{newstat}) therefore admits the explicit form
\begin{equation} f_{\gamma_0}'(t)=2\int_0^t \left<A(\gamma (t)),A(\gamma (u))\right>du. \label{firstf}\end{equation}
The components  $A^i(\gamma (u))$  ($i=1,..,d$) of the vector field $A(\gamma(u))$
 can be expanded  as
\begin{equation}\label{Corr-eq}A^i(\gamma (u))= A^i(\gamma (t))+\left<\nabla A^i(\xi_{i,u,t} ),\gamma (u)-\gamma(t)\right>\end{equation}
where, according to the Lagrange for of the remainder  of the $\mathbb{R}^d$ Taylor expansion, 
\begin{equation}\label{Corr-eq2} \xi_{i,u,t}=   \gamma (t)+\theta_i (\gamma (u)-\gamma(t))\quad \mbox{with $ \theta_i\in [0,1]$.
}
\end{equation}
Plugging (\ref{Corr-eq}) in the right-hand side of (\ref{firstf}), a trivial computation leads to
\begin{equation}\label{3.3.2}f_{\gamma_0}'(t)=2\int_0^t \sum_{i=1}^d\Big(|A^i(\gamma (t))|^2+A^i(\gamma (t))\left<\nabla A_i(\xi_{i,u,t}),\gamma (u)-\gamma(t)\right>\Big)du. \end{equation}
The proof the theorem  ends proving that there exists 
 $T>0$ such that,
if  $0\leq  t\leq T$,  then 
\begin{equation}\label{3.3.3}
\sum_{i=1}^d|A^i(\gamma (t))\left<\nabla A^i(\xi_{u,t}),\gamma (u)-\gamma(t)\right>|\stackrel{\textrm{wish to prove}}{<} 
  \sum_{i=1}^d  |A^i(\gamma (t))|^2=\|A(\gamma (t))\|^2.
\end{equation}
Indeed, (\ref{3.3.3}) entails that the integrand in \eqref{3.3.2} -- that is the one  in (\ref{firstf}) -- is  strictly positive so that (\ref{newstat}) is valid
because the integrand of (\ref{firstf})  is also $u$-continuous.
To prove  (\ref{3.3.3}), let us focus on its left-hand side.
It is bounded by 
\begin{multline}\label{3.3.3-1}
\sum_{i=1}^d|A^i(\gamma (t))\left<\nabla A^i(\xi_{u,t}),\gamma (u)-\gamma(t)\right>|
\leq \sum_{i=1}^d|A_i(\gamma (t))|\ |\left<\nabla A^i(\xi_{u,t}),\gamma (u)-\gamma(t)\right>|\\
\leq   \sum_{i=1}^d  \|A(\gamma (t))\|  \|\nabla A^i(\xi_{u,t})\| \|\gamma (u)-\gamma(t)\|
\leq  d M_2   \|A(\gamma (t))\| \|\gamma (u)-\gamma(t)\|.
\end{multline}
The bound (\ref{3.3.3-1}) can be further improved 
estimating $ \|\gamma (u)-\gamma(t)\|$
with the following argument where  we 
use  the notation $\gamma(t)=\sum_{i=1}^d\gamma^i(t)e_i$ and 
we exploit again (\ref{easyEq}) and (\ref{Corr-eq})-(\ref{Corr-eq2}).
\begin{align*}
\gamma_i(u)-\gamma _i(t)&=\int_t^uA_i(\gamma (s))ds
=\int_t^uA^i(\gamma (t))ds+\int_t^u \left<\nabla A^i(\xi_{i,s,t}),\gamma (s)-\gamma (t)\right>ds\\
&= A^i(\gamma (t)) (u-t)+\int_t^u\left<\nabla A^i(\xi_{i,s,t}),\gamma (s)-\gamma (t)\right>ds. \label{in-45}
\end{align*}
Since $\|\nabla A_i(x)\|\leq M_2$ due to condition 2, we therefore have
\begin{equation*}
|\gamma_i (u)-\gamma_i(t)|\leq \|A(\gamma (t)\|(t-u)+\int_u^t M_2\|\gamma (s)-\gamma (t)\|ds,
\end{equation*}
so that
\begin{equation}\label{ineqforiter}
    \|\gamma (u)-\gamma(t)\|\leq \sqrt d\left(\|A(\gamma(t))\|(t-u)+\int_u^t M_2\|\gamma (s)-\gamma (t)\|ds\right).
\end{equation}
Let us iterate this inequality for $\|\gamma (u)-\gamma(t)\|$ finding an improved 
estimate in terms of $\|A(\gamma(t))\|$ and $t-u$, hence in terms of $T$ because $0\leq u\leq t\leq T$. 
Let us start by  applying  inequality (\ref{ineqforiter}) to the term $\|\gamma (s)-\gamma (t)\|$ on the 
integrand in the right-hand side of  (\ref{ineqforiter}):
\begin{multline*}
    \|\gamma (u)-\gamma(t)\|\leq \sqrt d\|A(\gamma(t))\|\left((t-u)
    +M_2
        \sqrt d\int_u^t(t-s_1)ds_1\right)\\
        +(M_2
        \sqrt d)^2\int_u^t \int_{s_1}^t
        \|\gamma (s_2)-\gamma (t)\|ds_2
            ds_1.
\end{multline*}
Applying (\ref{ineqforiter}) again, we obtain
\begin{multline*}
    \|\gamma (u)-\gamma(t)\|\leq \sqrt d\|A(\gamma(t))\|\left((t-u)
    +M_2
        \sqrt d\int_u^t(t-s_1)ds_1+(M_2
        \sqrt d)^2\int_u^t \int_{s_1}^t(t-s_2)ds_2ds_1\right)\\
               +(M_2
        \sqrt d)^3\int_u^t \int_{s_1}^t
        \int_{s_2}^t \|\gamma (s_3)-\gamma (t)\|ds_3
                ds_2
            ds_1.
\end{multline*}
To state the general estimate, let us introduce the $n$-dimensional orthogonal simplex,
$$\Delta_n:=\{(s_1,...,s_n)\in [u,t]^n\colon u\leq s_1\leq...\leq s_n\leq t\}$$
which is the  a corner of $n$-dimensional cube $[u,t]^n$ with length of the corner-touching edges $t-u$. 
The $n$-dimensional Lebesgue measure of $\Delta_n$ equals $|\Delta_n|=(t-u)^n/n!$ and a direct calculation shows that 
$|\Delta_{n+1}|=\int_{\Delta_{n}}(t-s_n)ds_n\dots ds_1$. With this notation, the last inequality reads
\begin{multline*}
    \|\gamma (u)-\gamma(t)\|\leq \sqrt d\|A(\gamma(t))\|\left(|\Delta_1|
    +M_2
        \sqrt d|\Delta_2|+(M_2
        \sqrt d)^2|\Delta_3|\right)\\
               +(M_2
        \sqrt d)^3\int_{\Delta_3}  \|\gamma (s_3)-\gamma (t)\|ds_3
                ds_2
            ds_1.
\end{multline*}
Applying (\ref{ineqforiter}) to this inequality as many times as we need for each $n\geq 1$, and recalling that  
$M_2\neq 0$, we have
\begin{multline*}
    \|\gamma (u)-\gamma(t)\|\leq \sqrt d\|A(\gamma(t))\|\frac{1}{M_2
        \sqrt d}\left(M_2
        \sqrt d|\Delta_1|
    +(M_2
        \sqrt d)^2|\Delta_2|+\dots+(M_2
        \sqrt d)^n|\Delta_{n}|\right)\\
               +(M_2
        \sqrt d)^{n+1}\int_{\Delta_{n}}  \|\gamma (s_{n})-\gamma (t)\|ds_{n}
                \dots 
            ds_1,
\end{multline*}
which, after exploiting  $|\Delta_n|=(t-u)^n/n!$, becomes
\begin{multline}\label{alslastest}
    \|\gamma (u)-\gamma(t)\|\leq \frac{\|A(\gamma(t))\|}{M_2}\sum_{m=1}^n\frac{\left(M_2
        \sqrt d(t-u)\right)^n}{n!}
               +(M_2
        \sqrt d)^{n+1}\int_{\Delta_{n}}  \|\gamma (s_{n})-\gamma (t)\|ds_{n}
                \dots 
            ds_1.
\end{multline}
An estimate of the remainder in this formula arises 
from  the trivial bound
\begin{equation}\label{derivedfrom3.3.1}
\|\gamma (t)-\gamma (u)\|=
\left\|\int _u^t A(\gamma(s))ds\right\|
\leq \int _u^t \|A(\gamma(s))\|ds
\leq M_1(t-u)\quad \mbox{if $0\leq u\leq t$}
\end{equation}
which specializes to  $\|\gamma (s_{n})-\gamma (t)\|\leq M_1(t-s_n)$ in (\ref{alslastest}), yielding
\begin{multline*}
\left|(M_2 \sqrt d)^{n+1}\int_{\Delta_{n}}  \|\gamma (s_{n})-\gamma (t)\|ds_{n}\dots ds_1\right|
\leq (M_2 \sqrt d)^{n+1}\int_{\Delta_{n}}  M_1(t-s_n)ds_{n}\dots ds_1\\
=M_1(M_2 \sqrt d)^{n+1}|\Delta_{n+1}|=\frac{M_1\left(M_2 \sqrt d(t-u)\right)^{n+1}}{(n+1)!}\longrightarrow 0 \textrm{ as }n\to\infty.
\end{multline*}
As a consequence, taking the limit as $n\to\infty$  in (\ref{alslastest}), we finally obtain
\begin{equation}\label{alslastest-1}
\|\gamma (u)-\gamma(t)\|\leq \frac{\|A(\gamma(t))\|}{M_2}\left(e^{M_2\sqrt d(t-u)}-1\right).
\end{equation}
 Combining (\ref{alslastest-1}) with (\ref{3.3.3-1}) we find
\begin{multline*}
\sum_{i=1}^d|A^i(\gamma (t))\left<\nabla A^i(\xi_{u,t}),\gamma (u)-\gamma(t)\right>|
\leq  d M_2   \|A(\gamma (t))\| \|\gamma (u)-\gamma(t)\|\\
\leq
d M_2\|A(\gamma (t))\|\frac{\|A(\gamma(t))\|}{M_2}\left(e^{M_2\sqrt d(t-u)}-1\right)=d\|A(\gamma (t))\|^2\left(e^{M_2\sqrt d(t-u)}-1\right)\\
\leq d\|A(\gamma (t))\|^2\left(e^{M_2\sqrt dT}-1\right)\:,
\end{multline*}
where at the last step we used the fact that  that $\mathbb{R}\ni y\longmapsto e^y$ is monotonically increasing and that 
 $t-u\leq T$ because $0\leq u\leq t\leq T$.
In summary,  we have established that, for all $T>0$, if  $0\leq u\leq t\leq T$, then
\begin{equation}\label{semifinalest-1}
    \sum_{i=1}^d|A^i(\gamma (t))\left<\nabla A^i(\xi_{u,t}),\gamma (u)-\gamma(t)\right>|
\leq d\|A(\gamma (t))\|^2\left(e^{M_2\sqrt dT}-1\right).
\end{equation}
This inequality is sufficient to prove  (\ref{3.3.3}) concluding the proof, just choosing 
 $T>0$ such that \begin{equation}\label{semifinalest-11}d\|A(\gamma (t))\|^2\left(e^{M_2\sqrt dT}-1\right)<\|A(\gamma (t))\|^2\quad \mbox{if $0\leq  t \leq T$}\:.\end{equation} 
This is always feasible because, as observed at the beginning of the proof,
$A(\gamma(t)) \neq 0$ if $A(\gamma(0)) \neq 0$ as we supposed in (\ref{newstat}).
We can therefore divide both sides of (\ref{semifinalest-11})
for $||A(\gamma(t))|| \neq 0$, and the resulting inequality is solved as (taking the constraint $T>0$ into account), 
\begin{equation}
0< T < \frac{1}{M_2\sqrt{d}}\ln \left( 1+ \frac{1}{d}\right) \label{T}
\end{equation}
Notice that this $T$ can be chosen independent of $\gamma_0=\gamma(0)$.
\end{proof}

This result can be extended to  Riemannian manifold $(M,g)$ of bounded geometry. Indeed, in this case the following result allows to prove a bound for  the euclidean norm of the components of a vector field $A$ annd of its covariant derivative $\nabla A$ in local normal charts in terms of their Riemannian norm $\|A\|_g$ and $\|\nabla A\|_g$.

\begin{proposition}\label{teo2} Let $(M,g)$ be a $d$-dimensional  smooth Riemannian manifold of bounded geometry. If $r_0 \in (0,I_{(M,g)})$ is sufficiently small, then there exist four constants $k_1,k_2,k_3.k_4 \in [0,+\infty)$ such that for every  local  normal Riemannian chart centered at every $p\in M$ 
$(B^{(M,g)}_{r_0}(p), \exp^{-1}_{p,N})$ with 
coordinates $y^1,\ldots,y^n$ and  every smooth vector field  $A$ on $M$, the following uniform bounds hold:
\begin{itemize}
\item[(a)]  $||A(y(q))||^2 \leq k_1 ||A(q)||^2_g$ \:,
\item[(b)]  $||\nabla A(y(q))||^2 \leq k_2 ||\nabla^{(g)} A(q)||^2_g  + k_3 ||A(q)||^2_g + k_4 ||\nabla^{(g)}  A(q)||_g  ||A(q)||_g$ \:,
\end{itemize}
when $q\in B^{(M,g)}_{r_0}(p)$  (i.e. $y(q)\in B_{r_0}(0) \subset \bR^d$).
\end{proposition}
\noindent Above $\nabla$ denotes the standard gradient in $\bR^d$  and $||\cdot ||$ indicates the standard pointwise Euclidean norm of vectors  and  $\bR^d$-$(1,1)$ tensors referring to their components in Cartesian coordinates $y^1,\ldots, y^d$:
$$||A(y)||^2 = \sum_{a=1}^d |A^a(y)|^2\:\quad \mbox{and}\quad ||T(y)||^2 := \sum_{a,b=1}^d |T_b^a(y)|^2\:,$$
whereas $||\cdot ||_g$ denotes the previously defined natural point-wise norm associated to the metric $g$ acting on vector fields and  tensor fields of order $(1,1)$ and $\nabla^{(g)}$ is the Levi-Civita covariant derivative associated to the metric.

\begin{proof}
See the appendix
\end{proof}

We are now in a position to state the final result which extends proposition \ref{TRd} to Riemannian manifolds of {bounded} geometry.

\begin{proposition}\label{teo4.5}
Let $(M,g)$ 
be a smooth    Riemannian manifold of bounded geometry (thus geodesically complete) and $A$ a smooth vector field on $M$ such that, for some  $c_1,c_2 \in (0,+\infty)$,
\begin{enumerate}
    \item $\sup_{x\in M}\|A(x)\|_g \leq  c_1$,
    \item $\sup_{x\in M}\|\nabla^{(g)} A\|_g \leq c_2 $.
\end{enumerate}
Consider the unique maximal and complete (for 1) smooth  solution $\gamma\colon \bR \to M$ of the Cauchy problem 
\begin{equation}\label{CPZn2}
\left\{\begin{array}{l}
     \dot \gamma (t)=A(\gamma(t))\\
     \gamma (0)=\gamma_0
\end{array}\right.
\end{equation}
for every $\gamma_0 \in M$ and define $d_{\gamma_0}:[0,+\infty) \to \bR$  as $$d_{\gamma_0}(t):= d_{(M,g)}(\gamma(0), \gamma(t)).$$
Then,  there exists  $T>0$ independent of $\gamma_0$ such that  the function $d_{\gamma_0}$ is   non-decreasing in $[0,T]$. 
Even more, $d_{\gamma_0}$ is strictly increasing in $[0,T]$  if $A(\gamma_0)\neq 0$. 
\end{proposition}

\begin{proof} 
First of all, exactly as for the case $M=\bR^d$, we remark that if $A(\gamma(0))=0$ then
 $d_{\gamma_0}(t)=d_{(M,g)}(\gamma(0), \gamma(t)) =0$ 
and the result holds trivially for any $T>0$. 
Let us therefore  restrict ourselves to the case  $A(\gamma(0))\neq 0$ where, by the local uniqueness 
of the solutions of the Cauchy problem  \eqref{CPZn2}, 
we have that $A(\gamma(t))\neq 0$ for all $t\neq 0$.
 Let $f_{\gamma_0}\colon[0,+\infty) \to \mathbb{R}$ be the  smooth map 
$f_{\gamma_0}(t)=d_{\gamma_0}(t)^2$. 
To prove the   thesis it is enough to demonstrate  that
\begin{equation} \mbox{\em if $A(\gamma_0)\neq 0$, then  there exists $T>0$ independent of $\gamma_0$ such that 
 $f_{\gamma_0}'(t)>0$ for all $t\in (0,T]$.} \label{newstat2} \end{equation}
Statement (\ref{newstat2}) will be demonstrated by reducing to the analogous proof in $\bR^d$ here 
performed in a suitably Riemannian coordinate patch centered on $\gamma(0)$. To this end it is
 fundamental to prove that the solution $\gamma(t)$ cannot exit such a Riemannian coordinate domain.
For a given $\gamma(0)\in M$ take $r\in (0, I_{(M,g)})$ and consider the geodesical ball $B_r^{(M,g)}(\gamma(0)))$. 
We prove that there is $T'>0$, independent of $\gamma(0)$,
such that $\gamma(t) \in B_r^{(M,g)}(\gamma(0)))$ for $t\in [0,T']$.
From the definition (\ref{defd}) of $d_{(M,g)}$ we have that 
$$d_{(M,g)}(\gamma(T'),\gamma(0)) \leq \int_0^{T'} \|\dot{\gamma}(t)\| dt = 
\int_0^{T'} \|A(\gamma(t))\| dt \leq  \int_0^{T'}  c_1 dt =T' c_1\:.$$
We conclude that, defining $T' := r/c_1$, we have that $\gamma(t) \in B_r^{(M,g)}(\gamma(0)))$ for $t\in [0,T']$ as wanted. 
We henceforth restrict our attention to the ball $B_r^{(M,g)}(\gamma(0)))$, since the curve cannot exit it if $t\in [0,T')$,
 looking for $T \in (0,T')$ satisfying (\ref{newstat2}). We can describe the curve $\gamma$ in Riemannian coordinates 
$y^1,\ldots, y^d$ centered on $\gamma(0)$ inside the ball $B_r(0)\subset \bR^d$, taking advantage of the results
already proved in $\bR^d$ in proposition \ref{TRd}. Now, the crucial observation is that, due to (\ref{distanceG}) 
and noticing that $\gamma(0)$ coincides to the origin $0$ of $\bR^d$ when describing it in Riemannian
 coordinates $y^1,\ldots, y^d$, we have that
$$d_{\gamma(0)}(t) = ||\gamma(t)-\gamma(0)||\:,$$
where the norm is the Euclidean one in $\bR^n$ when describing the curve $\gamma$ in coordinates
 $\gamma(t) \equiv (y^1(t), \ldots, y^d(t))$. From now on the proof of (\ref{newstat2}) is identical to that
 of (\ref{newstat}), using the fact that, in the said coordinate patch, 
conditions 1 and 2 in proposition \ref{TRd} are true for $x \in B_r(0)$ if choosing the initial $r=r_0$ sufficiently
 small that proposition \ref{teo2} is valid (observe that this choice is independent of $\gamma(0)$).
As a matter of fact, with the said $r_0$, taking advantage of (a) and (b) in proposition \ref{teo2}, we can choose
$$M_1\geq \sqrt{k_1}c_1\quad \mbox{and}\quad M_2 \geq \sqrt{k_2c_1^2+ k_3c_2^2 + k_4c_1c_2}\:.$$
With the proof of proposition \ref{TRd} and $M_1,M_2$ as above  (taking $M_2>0$ as in the proof of proposition  \ref{TRd}), the wanted
 $T$ is every $T \in (0,T')$ which also satisfies (\ref{T}). It is clear from the procedure that $T$ can be chosen independent of $\gamma(0)$.
\end{proof}

\subsubsection{Weak convergence of the sequence $Z_n$ to $X$ }
Coming back to the sequence $\tilde Z_n$ of random walks defined in \eqref{deftildeZ-n}, the results of proposition \ref{teo4.5} allow to prove that for any $T>0$ the sequence of measures $\tilde \mu _n$ on $(C([0,t], M), \cB(C([0,t], M))$ induced by $\tilde Z_n$ converges weakly to the measure $\mu$ induced by the diffusion process $X$.
\begin{theorem}\label{teo-conv-lastZn}

Under the assumptions of theorem \ref{teo3.1}, the sequence of measures $\tilde\mu _n$ on $(C([0,t], M), \cB(C([0,t], M))$ induced by the random walks $\tilde Z_n$ defined by \eqref{deftildeZ-n} converges weakly to the measure $\mu$ on $(C([0,t], M), \cB(C([0,t], M))$  induced by the diffusion process $X$ associated with the elliptic operator $L$.  
\end{theorem}
\begin{proof}

Since by assumptions $(M,g)$ is of bounded geometry and the vector fields $\{A_k\}_{k=0,\ldots , r}$ are $C^\infty$-bounded, they satisfy the assumptions of proposition \ref{teo4.5}. In particular there exists two constants $c_1,c_2\in \bR^+$ such that for all $k=0,\dots, r$
 $$\sup_{x\in M}\|A_k(x)\|_g \leq  c_1,\quad
    \sup_{x\in M}\|\nabla^{(g)} A_k\|_g \leq c_2, $$
and there exists a $T>0$ such that for all $k=0,\dots, r$ and $x\in M$ the functions $d^k:\bR^+\to\bR$ defined as 
$d^k(t):=d(x, \gamma_{x,\pm A_k}(t)$ in non-decreasing  for $t\in [0,T]$, with $\gamma_{x,A}$ denoting the maximal solution of the Cauchy problem \eqref{CP-M}.

The main argument  is now completely similar to the one in the proof of theorem \ref{teoconvC}.
 Let us consider the trajectories $\gamma_\omega$ of the process $\tilde Z_n$, defined as $\gamma_\omega (t):=\tilde Z_n(t)(\omega)$. 
 By proposition \ref{teo4.5} there exists $T>0$ such that for any $x\in M$ we have $d(x,\gamma_x(t))\leq d(x,\gamma_x(t'))$ for all $0\leq t\leq t'\leq T$, with $\gamma _x:[0,+\infty )\to M$ is the maximal solution of the Cauchy problem \eqref{CPZn2}.
 Fix $\delta >0$ and take $n$ sufficiently large in such a way that $1/n<\min (\delta, T)$ and . Consider $s,t\in [0,T]$, $s<t$, $|t-s|<\delta$. We will have   $s\in [m/n,(m+1)/n]$ and $t\in [m'/n,(m'+1)/n]$, with $m\leq m'$, hence:
\begin{align*}
    & d(\gamma_\omega (s), \gamma _\omega (t))\\
    &\leq
    d\left(\gamma_\omega(s),\gamma_\omega((m+1)/n)\right)+d\left(\gamma_\omega((m+1)/n),\gamma_\omega(m'/n)\right)+ d\left(\gamma_\omega(m'/n),\gamma_\omega(t)\right)\\
    &\leq d\left(\gamma_\omega(m/n),\gamma_\omega((m+1)/n)\right)+d\left(\gamma_\omega((m+1)/n),\gamma_\omega(m'/n)\right)+ d\left(\gamma_\omega(m'/n),\gamma_\omega((m'+1)/n)\right)\\
    &\leq 3\max\{ d\left(\gamma_\omega(m/n),\gamma_\omega(m'/n)\right), |m/n-m'/n|<\delta\}
\end{align*}

The probability that the modulus of continuity of the trajectories of $\tilde Z_n$ exceeds a given $\varepsilon >0$ can be estimated by 
\begin{eqnarray*}
&&\mu_n \left(\{\gamma\in C_M[0,T]\colon w(\gamma ,\delta)>\varepsilon\}\right)\\
&&\leq \mu_n \left(\{\gamma\in C_M[0,T]\colon \max_m\{d(\gamma (m/n),\gamma (m+1)/n))\}>\varepsilon/3\}\right)\\
& &=\mu_{X_n}\left(\{\gamma\in D_M[0,T]\colon
w(\gamma ,\delta)>\varepsilon /3\}\right)
\end{eqnarray*}
By theorem \ref{thconvD} and lemma \ref{lemma4}, we get for any $\varepsilon>0$
$$\lim_{\delta\downarrow 0}\limsup_n \mu_n \left(\{\gamma\in C_M[0,T]\colon w(\gamma ,\delta)>\varepsilon\}\right) =0$$
Since $\tilde Z_n(0)=x$ for any $n$, the sequence of probability measures $\{\mu_n\}$ is tight \cite{Bil} and the measure $\mu$, i.e. the law of $X$ is the only possible limit point.
\end{proof}

\section{Heat equation and  Brownian motion on parallelizable manifolds  }\label{sez5}

The results of the previous sections can be also applied to the construction on the Brownian motion on $M$. Here we shall assume that the manifold $M$ is   {\bf parallelizable} i.e. that there exist smooth vector fields $\{e_k\}_{k =1,...,d}$ such that for any $x\in M$ the vectors $\{e_k \}_{k=1,...,d}$ provide a linear basis of $T_xM$.  
Examples of such manifolds are e.g. the spheres $S^1$, $S^3$, $S^7$ and Lie groups as well as orientable 3-manifolds.
Without loss of generality, we can take 
$\{e_k \}_{k =1,...,d}$ in such a way that for any $x\in M$ the vectors $\{e_k\}_{k=1,...,d}$ are orthonormal with respect to the metric tensor $g$. Further, given a local neighborhood $U$, the components $e_k^i$ the vectors $e_k$ with respect to the local basis $\partial _i:=\frac{\partial}{\partial x^i}$   satisfy the following equality:
$$\sum _{k=1}^de_k^i(x)e_k^j(x)=g^{ij}(x)$$

Let us consider the  Laplace-Beltrami operator $L_0:=\Delta_{LB}$ on $M$ defined in local coordinates on the smooth maps $u\in C^\infty(M)$ as:
$$ \Delta_{LB}u=\sum_{i,j=1}^dg^{ij}\nabla^{(g)}_i\nabla^{(g)}_ju\:,$$
or, more explicitly
$$(\Delta_{LB}u)(x)=\sum _{i,j=1}^dg^{ij}(x)\left(\frac{\partial^2u}{\partial x^i\partial x^j}(x)-\sum_{k=1}^d\Gamma_{ij}^k\frac{\partial u}{\partial x^k}(x)\right).$$

Under suitable hypotheses,  the results of previous sections can be applied to  $\Delta_{LB}$ providing on the one hand the existence of an associated  Feller semigroup  - the {\em heat semigroup} - in $C_0(M)$ and, on the other hand, a Chernoff approximation in terms of translation operators of the form \eqref{S(t)} or \eqref{eqshift1-hm}. From the probabilistic point of view, these results yield also an approximation for the Brownian motion on $M$, i.e. the diffusion process associated to the heat semigroup, in terms of the weak limit of sequences of different types of random walks on $M$.

More precisely we have the following result.

\begin{theorem} Let $(M,g)$ be a smooth    Riemannian manifold of bounded geometry.
Then the closure in $C_0(M)$ of $\Delta_{LB}|_{D_k}$ where $D_k$ is defined in (\ref{domainD})  with $L_0:= \frac{1}{2}\Delta_{LB}$ is the generator of a (unique) Feller semigroup on $C_0(M)$. Both the generator and the semigroup are independent of $k=0,1, \ldots$. 
\end{theorem}
\begin{proof} Since $(M,g)$ is of bounded geometry $-\Delta_{LB}$ is $C^\infty$-bounded, furthermore 
$\Delta_{LB}|_{C_c^\infty}$ is symmetric and $-\Delta_{LB}|_{C_c^\infty}\geq 0$. 
Finally $-\Delta_{LB}$ is  automatically uniformly elliptic since the matrix defining its pricipal symbol is nothing but the metric $g$.
Hence $\Delta_{LB}$ enjoys exactly the same properties as those of the operator $L_0$ we used in the proof of 
lemma \ref{propL} and
proposition  \ref{teoL}.
The proof for $\Delta_{LB}$ is therefore identical.
\end{proof}
\subsection{An approximation  in terms of random walk with piecewise geodesic paths }
\begin{lemma}\label{lemma-fin}Let $(M,g)$ be a smooth parallelizable Riemannian manifold of bounded geometry.

  For each $x\in M$, $t\geq 0$, $f\in C_0(M)$  set
 \begin{equation}\label{eqshift1-hm}
 (S(t)f)(x)=\frac{1}{2d}\sum_{k=1}^d\bigg(f\left(\gamma_{x,\sqrt{d }e_k}(\sqrt t)\right)+f\left(\gamma_{x,-\sqrt{d }e_k}(\sqrt t)\right)\bigg)
\end{equation}
where $\gamma _{x,v}$ denotes the geodesics starting at time 0 at the point $x\in M$ with initial velocity $v\in T_xM$. Further let $ L_0:C^\infty(m)\to C^\infty(M)$ be the differential operator 
$ L _0=\frac{1}{2}
\Delta_{LB}$ and let        $L_1:=L_0|_D$, where $D$ is given by \eqref{domainD}. \\
Then, with respect to the norm $\|f\|=\sup_{x\in M}|f(x)|$, the following holds:
\begin{itemize}
\item[(I)] for each $t\geq 0$ and $f\in C_0(M)$ we have $S(t)f\in C_0(M)$ and $\|S(t)f\|\leq  \|f\|$.

\item[(II)] for each $f\in D_k$, with $k\geq 3$,  we have  $\lim_{t\to+0}\|S(t)f-f-tL_1f\|/t=0$.

\item[(III)] if $t\to t_0$, $t_n\geq 0$ and $f\in C_0(M)$, then $\lim\limits_{t\to t_0}\|S(t)f- S(t_0)f\|=0$ for each $t_0\geq 0$.
 \end{itemize}

 \end{lemma}
 
 \begin{proof}
 First of all we remark that under the stated assumptions the manifold is geodesically complete. Indeed, this follows from the  bounded geometry assumption and lemma \ref{lemmaC}.\\
 The proof of I) and III) is completely analogous to the proof of points 2., 3a. and 3b. of theorem \ref{teo3.1}. We can restrict ourselves to prove point II).\\
 for $t\downarrow 0$, we have
 $$f(\gamma _{x,v}(t))=f(x)+vf(x)t+\frac{1}{2}\frac{d^2}{ds^2}f(\gamma _{x,v}(s))_{|s=0}t^2+\frac{t^3}{3!}R(t,x),$$
 with $R(t,x)=\frac{d^3}{ds^3}f(\gamma _{x,v}(s))_{|s=u}$, $u\in [0,t]$.
 In particular,  by the geodesic equation 
 \begin{equation}\label{geod-eq}
     {\ddot\gamma}^k _{x,v}(t)=-\Gamma_{ij}^k\dot\gamma^i _{x,v}(t)\dot\gamma^j _{x,v}(t), 
 \end{equation}
  we obtain
 \begin{align*}
 \frac{d^2}{dt^2}f(\gamma _{x,v}(t))&=\sum_{i,j}\partial ^2_{ij}f(\gamma _{x,v}(t))\dot\gamma^i _{x,v}(t)\dot\gamma^j _{x,v}(t)+\sum_{i}\partial _{i}f(\gamma _{x,v}(t))\ddot\gamma^i _{x,v}(t),\\
 &=\sum_{i,j}\partial ^2_{ij}f(\gamma _{x,v}(t))\dot\gamma^i _{x,v}(t)\dot\gamma^j _{x,v}(t)-\sum_{i,j,k}\partial _{k}f(\gamma _{x,v}(t))\Gamma_{ij}^k\dot\gamma^i _{x,v}(t)\dot\gamma^j _{x,v}(t).\\
 \end{align*}
 Analogously, 
 \begin{equation}\label{trdDerivative}\frac{d^3}{dt^3}f(\gamma _{x,v}(t))=
\left((2\Gamma_{mj}^i\Gamma_{lk}^m-\partial _l\Gamma_{kj}^i)\partial _i f+\partial _{lkj}f+3\Gamma^i_{kl}\partial _{ij}f\right)
\dot\gamma^l _{x,v}(t)\dot\gamma^k _{x,v}(t)\dot\gamma^j _{x,v}(t),
 \end{equation}
 (where, for notational simplicity, we have used the convention on the sum over repeated indices). 
Hence, by using the identity $\sum_k e_k ^ie_k ^j=g(x)^{ij}$: 
 \begin{align*}
 S(t)f(x)&=
 f(x)+\frac{1}{2}\sum_{k=1}\left(\sum_{i,j}\partial ^2_{ij}f(x)e_k^ie_k ^j -\sum_{k,i,j}\partial _{k}f(x)\Gamma^k_{ij}e_k^ie_k ^j\right)t+t^3/2R(t,x)\\
 &=
 f(x)+L_1f(x)+t^3/2R(t,x),\\
 \end{align*}
 with 
 $$R(t,x)= \frac{1}{12d}\sum_{k=1}^d \left(\frac{d^3}{dt^3}f(\gamma _{x,\sqrt d e_k}(t))_{|t=u_k}+\frac{d^3}{dt^3}f(\gamma _{x,-\sqrt d e_k}(t))_{|t=u'_k}\right)$$
with $u_k, u'_k\in [0, \sqrt d]$, $k=1, \ldots , d$,  and $\frac{d^3}{dt^3}f(\gamma _{x,\sqrt d e_k}(t))$ is given by \eqref{trdDerivative}.\\
Let us take an $r_0\in (0, I_{(M,g))}]$ sufficiently small in such a way that the thesis of proposition \ref{teo2} holds and consider an atlas made of local normal Riemannian charts $(B^{(M,g)}_{r_0}(p), \exp^{-1}_{p,N})$. 
By the assumption that $(M,g)$ is of bounded geometry,  estimate \eqref{estimateg3}, the bound
$$|\dot \gamma_{x,v}^i(t) |\leq \sqrt{\sum_{i=1}^d |\dot  \gamma_{x,v}^i(t) |^2} \leq k_1 \|v\|_g$$
resulting from statement (a) of proposition \ref{teo2} and by the geodesic equation \eqref{geod-eq}, and the condition $f\in D_k$ with $k\geq 3$,
we obtain: $$\sup_{t\in [0,1], x \in M}|R(t,x)|<\infty,$$ which yields II.
 \end{proof}
 \begin{corollary}\label{corollaryheat}
 Under the assumptions of lemma \ref{lemma-fin} the closure in $C_0(M)$ of $L_1$ is the generator of a Feller semigroup $V$ and for any $f\in C_0(M)$ and $T>0$:
\begin{equation}\label{convergenceformulaheat}\lim_{n\to \infty}\sup _{t\in [0,T]}\|S(t/n)^nf-V(t)f\|=0\:.\end{equation}

 \end{corollary}
 The heat semigroup $V$  provides a solution of the heat equation on $M$ 
 \begin{equation}\label{CauchyProblemheat}
 \left\{ \begin{array}{l}
 \frac{\partial }{\partial t}u(t,x)=\frac{1}{2}\Delta_{LB}u(t,x)\\
 u(0,x)=u_0(x)
 \end{array}\right.\end{equation}
 in the sense that if $u_0\in D(L)$ then  $u(t):=V(t)u_0\in D(L)$ and $\frac{d}{dt}u(t)=Lu(t)$ in the strong sense.\par

 Analogously to the case of diffusion processes on manifolds, the approximation result stated in corollary \ref{corollaryheat} admits a probabilistic interpretation.
Indeed, we can still define a sequence of random walks on $M$ with steps given by geodesic arcs according to the following construction.
 
 For any $n\in \bN$,  let $X_n $ be a jump process defined as 
 $$ X_n(0)=x, \qquad X_n(t):=X_n(\lfloor nt\rfloor/n)=Y_n(\lfloor nt\rfloor),$$
 where $\{Y_n(m)\}_m$ is a Markov chain with transition probabilities
 \begin{multline} \bP(Y_n(m)\in I|Y_n(m-1)=y) =\frac{1}{2d}\sum_{k=1}^d\left(\delta_{\gamma_{y, \sqrt{d}e_k(y)}(\sqrt{1/n})}\left(I\right) + \delta_{\gamma_{y, -\sqrt{d}e_k(y)}(\sqrt{1/n})}\right)\left(I\right), \quad I\in \cB(M).\end{multline}
Analogously, let $(Z_n)$ the sequence of processes with continuous paths obtained by $X_n$ as geodesic interpolation, namely:
$$ Z_n(0)=x, \quad Z_n(m/n)=X_n(m/n), \quad Z_n(t)=\gamma _{X_n(m/n),X_n((m+1)/n)(t-m_n)}, \: t\in [m/n,(m+1)/n]$$
where $\gamma_{x,y}$ is the geodesic such that $\gamma_{x,y}(0)=x$ and  $\gamma_{x,y}(1/n)=y$.

Denoted with $X$ the diffusion process on $M$ associated to the semigroup generated by the operator $L=\bar L_1$ we have the following result
\begin{theorem}\label{teo43} Under the assumption of corollary \ref{corollaryheat},
for any $T>0$, $X_n$ converges weakly to $X$ in $D_M[0,T]$ and $Z_n$ converges weakly to $X$ in $C_M[0,T]$
\end{theorem}
The proof is completely similar to the proofs of theorems \ref{thconvD} and \ref{teoconvC}.
\subsection{An approximation  in terms of random walk with steps along integral curves of the parallelizing vector fields }
 In the case where  the parallelizing vector fields $e_1,\ldots,e_d$ of the manifold $(M,g)$ 
(simultaneously of bounded geometry and parallelizable) 
are $C^\infty$-bounded, we can 
view $\Delta_{LB}$ as a subcase of the operator $L_0$ discussed in Section \ref{sec3} and recast all the discussion therein using the paths constructed out of the integral 
lines of the fields $e_k$ instead of the geodesics. In fact, since $\sum_{i=1}^d e_i^a(x) e_i^b(x) = g^{ab}(x)$ and using the fact that 
$\nabla^{(g)}_k g^{ab} =0$, we can write
$$\Delta_{LB} = \sum_{a,b=1}^d g^{ij} \nabla_a^{(g)}\nabla_b^{(g)} = \sum_{a,b=1}^d \nabla_a^{(g)} g^{ab} \nabla_b^{(g)} 
=  \sum_{a,b=1}^d \nabla_a^{(g)} \sum_{i=1}^d e^a_i e^a_i \nabla_b^{(g)} =
 \sum_{i=1}^d  \sum_{a,b=1}^d \nabla_a^{(g)} e^a_i e^a_i \nabla_b^{(g)}  $$
$$= \sum_{i=1}^d  \sum_{a,b=1}^d e^a_i \nabla_a^{(g)}  e^a_i \nabla_b^{(g)} + \sum_{i=1}^d (\nabla^{(g)}\cdot e_i) e_i  $$
In other words $\Delta_{LB}$ is the operator $L_0$ in (\ref{Hoperator0}) generated by the vector fierlds $e_1,\ldots, e_d$, with a suitable choice for 
$e_0$ since,
 if $f \in C^\infty(M)$,
$$ (\Delta_{LB} f)(x) = \sum_{i=1}^d e_i(e_i f)(x)+ (e_0 f)(x) \quad 
\mbox{where}\quad e_0 := \sum_{i=1}^d (\nabla^{(g)}\cdot e_i) e_i\:.  $$

In this case theorem \ref{teo-conv-lastZn} holds yielding the Brownian motion on $M$, i.e. the diffusion process associated with the Laplace-Beltrami operator $\Delta_{LB}$, as the weak limit of a sequence of random walks $\{\tilde Z_n\}$ of the form \eqref{deftildeZ-n}, with steps constructed out of integral curve of the vector fields $\{e_k\}_{k=1,\ldots, d}$. This result can be rephrased in following form.
\begin{theorem}\label{teo44}
Let $(M,g)$ be a smooth parallelizable manifold of bounded geometry admitting a set of parallelizing vector fields $e_1,\ldots,e_d$ which are $C^\infty$-bounded. Then the Wiener measure $\mu$ on $(C([0,t], M), \cB(C([0,t], M))$, i.e. the law of the diffusion process associated to the Laplace Beltrami operator $\Delta _{LB}$ is the weak limit of the sequence of probability measures $\tilde\mu _n$ on $(C([0,t], M), \cB(C([0,t], M))$ induced by the random walks $\tilde Z_n$ defined by \eqref{deftildeZ-n} with $A_k=e_k$.
\end{theorem}

\section{Acknowledgments}
We are grateful to Sergio Albeverio, Fernanda Florido-Calvo,  Christian G\'erard,  Simone Murro, and Andrea Pugliese for useful  discussions,
 suggestions, and for having pointed out relevant references to us. The financial support of CIRM  (Centro Internazionale per la Ricerca Matematica) - FBK (Fondazione Bruno Kessler), is gratefully acknowledged. Ivan Remizov's work is partially supported by Laboratory of Dynamical
Systems and Applications NRU HSE, of the Ministry of science and higher education of
the RF grant ag. No 075-15-2019-1931.
O.G. Smolyanov acknowledges the financial support of the grant "Fundamental problems of mechanics and mathematics" of Lomonosov Moscow State University and also the financial support
of Moscow Institute of Physics and Technology, within the state program to support the leading Russian Universities.
 
 \section{Proof of some technical propositions}
 
\noindent{\bf Proof of Lemma \ref{lemmaC}}. Suppose there is a maximal geodesic $\gamma : I\ni t  \to \gamma(t) \in M$,
where $t$ is the a length parameter along $\gamma$ used as its affine parameter, 
such that $\sup I = \omega <+\infty$ (the case $-\infty < \inf I$ is analogous). Let $\{t_n\}_{n\in \mathbb{N}} \subset I$ be an increasing  sequence such that $t_n \to \omega$ as $n\to +\infty$.
Consider an element $t_n$.
 If there were an open
ball $B_n\subset T_{\gamma(t_n)}M$ centered at the origin and of  radius $r > \omega- t_n$ where the exponential map $\exp_{\gamma(t_n)} T_{\gamma(t_n)}B_n 
\to M$ 
is a diffeomorphism onto its image, then
$B_n$ would include in particular the tangent vector of $\gamma$ at $\gamma(t_n)$ and also a longer parallel vector. As a consequence 
 $\gamma$ could  be extended  to a longer geodesics. Since this is not possible, we conclude that  $I_{(M,g)}(\gamma(t_n)) < \omega- t_n$. In turn,  it would imply $0\leq I_{(M,g)} \leq \inf_{n \in \mathbb{N}}I_{(M,g)}(\gamma(t_n))=0$, whereas $I_{(M,g)}>0$ by hypothesis. Hence all maximal geodesics must be   complete. The last statement immediately arises from 
Hopf-Rinow's theorem.
$\hfill \Box$\\

 \noindent{\bf Proof of Lemma \ref{CPVF}}.
Let $\gamma:(a,b)\to M$ be a maximal solution of \eqref{CPVF} and let us assume {\em ab absurdum} that $b <+\infty$.  Consider a $t_0\in (a,b)$ and let $f:(t_0,b)\to \bR$ be the continuous function defined as $$f(t):=d(\gamma(t),\gamma(t_0))\, ,$$
where $d:= d_{(M,g)}$ is the above defined distance induced by the Riemannian metric. Since we have assumed that $ b<\infty$, the function $f$ cannot be bounded on $[t_0,b)$. Indeed, $f$ were bounded, then there would exist an $R>0$ such that $\gamma (t)\in B_R(\gamma(t_0))$ for all $t\in [t_0, b)$, where $B_R(\gamma(t_0))$ denotes the closed ball with radius $R$ and center $\gamma(t_0)$. On the other hand, under the stated assumptions on $M$, 
Hopf-Rinow theorem assures the compactness of the closed metric balls. By a classical result (see, e.g., lemma 56, Ch. 1 in \cite{ONe}), if there exists a compact set $K$ such that  the maximal solution $\gamma:[t_0,b)\to M$
satisfies the condition $\gamma ([t_0,b))\subset K$, then $b=+\infty$.
Hence, since $f$ cannot be bounded,  there exists a monotonically increasing sequence $t_n\to b$ such that $d(\gamma(t_n),\gamma (t_0))\to \infty $. Let 
$s\colon [t_0,b)\to \bR$ be the curvilinear abscissa along the curve $\gamma$, namely:
\begin{equation}\label{st} s(t)=\int_{t_0}^t\sqrt{g(A(\gamma (u)), A(\gamma(u)))}du. \end{equation}
Clearly, for any $n\geq 1$, the following holds
$$\frac{d(\gamma(t_n), \gamma (t_0))}{t_n-t_0}\leq \frac{s(t_n)-s(t_0)}{t_n-t_0}.$$
the latter inequality,   the boundedness of the sequence $\{t_n-t_0\}$ and the fact that $\{d(\gamma(t_n), \gamma (t_0))\}$ is unbounded and strictly positive gives $$\lim_{n\to\infty}\frac{s(t_n)-s(t_0)}{t_n-t_0}= +\infty$$
On the other hand, by Lagrange's theorem applied to the (known to be differentiable) function $s: [t_0, b )\to \bR$ defined in \eqref{st}, for any $n$ there exist a $u_n\in (t_0, t_n)$ such that $$\sqrt{g(A(\gamma (u_n)), A(\gamma(u_n)))}=\frac{d s}{dt}(u_n)=\frac{s(t_n)-s(t_0)}{t_n-t_0}.$$ 
The left hand side of the above equality is bounded by the assumptions on $A$, while the right hand side is unbounded by the discussion above and we have obtained a contradiction. $\hfill \Box$\\

\noindent {\bf Proof of Proposition \ref{corollaryD}}.
Let us start with the following lemma.

\begin{lemma}\label{lemmaDEN} Let $M$ be a smooth manifold  and $f \in C_0(M)$. For every $\varepsilon >0$ there is $\psi \in C^\infty(M) \cap C_0(M)$ such that $||f- \psi||_\infty <\varepsilon$.
\end{lemma}
\begin{proof} (There are different  ways to prove this density result and this is just a possibility).
It is sufficient to prove the thesis for
real functions and, in turn, for
$f\geq 0$. The general statement follows by decomposing $f = f_+ -f_-$ where $0\leq f_\pm = \frac{1}{2}(|f|\pm f) \in C_0(M)$. 
Let us therefore prove the thesis for $0\leq f \in C_0(M)$.\\
If $p\in M$, there is a local chart $(U,\psi)$ such that $p\in U$. We can always restrict $U$ to a smaller open neightborhood $V$ of $p$, such that 
$\overline{V} \subset U$ is a compact set. Since there is such a local chart for every $p\in M$ and the topology of $M$ is $2$nd countable, we can extract a  subcovering of $M$ made of charts $\{V_j, \psi_j\}_{j\in J}$ where $J$ is finite or countably infinite. Using paracompactness property of $M$, we can refine $\{V_j, \psi_j\}_{j\in J}$ to a locally finite  covering (equipped with corresponding coordinate maps $\psi_j$,  the restrictions of the original ones) still indicated with the same 
symbol $\{V_j, \psi_j\}_{j\in J}$. Finally, we can define a partition of the unit $\{\chi_j\}_{j\in J}$ subordered to the covering $\{V_j\}_{j\in J}$. Therefore
\begin{itemize}
\item[(i)] $\chi_j \in C_c^\infty(M)$, 
\item[(ii)] $0\leq \chi_j \leq 1$, 
\item[(iii)] $supp(\chi_j) \subset V_j$,

\item[(iv)] $\sum_{j\in J}\chi_j(x)=1$ where, due to locally finiteness property,  for every $x\in M$ there is an open set containing $x$ whose intesection with the $V_j$ is not empty only for a finite number of indices $j\in J$, hence the sum is always finite.
\end{itemize}
To go on we assume that $J=\bN$ (the case of $J$ finite is simpler).  If $f \in C_0(M)$, the function $f|_{V_n}\geq 0$ represented  in coordinates through the map $\psi_n$ turns out to be  the restriction of a  contiunuous function 
defined on a compact $\psi_n(\overline{V_n})\subset \bR^n$. Using Stone-Weierstrass theorem we conclude that, for every $\varepsilon>0$, there is a smooth function $p^{(n,\varepsilon)}$ defined on $V$ that, in coordinates is the restriction to $V$ of a polynomial defined in  the compact set $\psi_n(\overline{V_n})\subset \bR^n$, such that with obvious notation
\begin{equation}
||f|_{V_n} - p^{(n,\varepsilon)}||^{(V_n)}_\infty < \varepsilon\label{one}.
\end{equation}
It is always possible to choose 
\begin{equation}
0\leq p^{(n,\varepsilon)} \leq f|_{V_n}\label{two}.
\end{equation}
In fact, for $\mu>0$  define 
$g_\mu := f + \mu$. Using the same argument as above, there is a smooth function $ q^{(n,\mu)}$ (in coordinates the restriction to the compact 
$\psi_n(\overline{V_n})$ of a polynomial) such that the inequality holds $|| q^{(n,\mu)}  - g_\mu||_\infty < \mu/3$, that is if $x\in V_n$
$$-\mu/3 \leq    q^{(n,\mu)}(x)  - f(x) - \mu < \mu/3$$
which implies
$$2\mu/3 <  q^{(n,\mu)}(x)  - f(x) < 4\mu/3$$
so that 
$$0 < f(x) + 2\mu/3 < q^{(n,\mu)}(x) < f(x) + 4\mu/3$$
Defining $\varepsilon := 4\mu/3$ and $p^{(n,\varepsilon)}:= q^{(n,\mu)}$ we have that (\ref{one}) and (\ref{two}) are valid simultaneously.
In view of the definition of the functions  $\chi_n$, (\ref{one}) and (\ref{two}) immediately imply
\begin{equation}
||f \cdot \chi_n - p^{(n,\varepsilon)}\chi_n||_\infty < \varepsilon\label{one2}.
\end{equation}
and 
\begin{equation}
0\leq p^{(n,\varepsilon)} \cdot \chi_n \leq f\cdot \chi_n \label{two2}.
\end{equation}
Notice that the functions $ p^{(n,\varepsilon)} \cdot \chi_n$ and $f\cdot \chi_n $ are everywhere well defined on $M$ and belong to $C^\infty_c(M)$.
To conclude the proof, for $\varepsilon>0$ define
$$\psi := \sum_{n\in \bN} \chi_n \cdot  p^{(n,\varepsilon/2^{n+1})} $$
This function  is well-defined belongs to $C^\infty(M)$. Furthermore 
$$0\leq \psi = \sum_{n\in \bN} \chi_n \cdot  p^{(n,\varepsilon/2^{n+1})}  \leq  \sum_{n\in \bN} \chi_n \cdot  f = f $$
so that $\psi \in C^\infty(M) \cap C_0(M)$. Finally
$$||f-\psi||_\infty = \left|\left| \sum_{n\in \bN} \chi_n \cdot  p^{(n,\varepsilon/2^{n+1})}- \chi_n \cdot f \right|\right|_\infty \leq   \sum_{n\in \bN} || \chi_n \cdot  p^{(n,\varepsilon/2^{n+1})}- \chi_n \cdot f ||_\infty \leq \sum_{n\in \bN} \varepsilon 2^{n+1} = \varepsilon\:.$$
\end{proof}
In view of the lemma, in turn, it is sufficient to prove that $C_c^\infty(M)$ is dense in $C_0(M) \cap C^\infty(M)$.  If   $f\in C_0(M) \cap C^\infty(M)$ and $\varepsilon>0$, 
then there is a compact $K \subset M$ such that $|f(x)| <\varepsilon$ if $x \not \in K$. Let $A\supset K$ be an open set  whose closure is compact 
(It can be constructed as follows. Every $p\in K$ admits an open neighborhood which is relatively compact -- just work in a coordinate patch-- due compactness ,  $K$ is therefore covered by a finite class of those relatively-compact open sets. The union of those sets is the wanted $A$.) Define $B := M \setminus A$.
Since $K$ and $B$ are disjoint closed sets ($K$ is closed because $M$ is Hausdorff by hypothesis),   from the {\em smooth Urysohn lemma}, there exists $\chi \in C^\infty(M)$ such that $|\chi(x)| \leq 1$ for $x\in M$ and
$K\subset\chi^{-1}(\{1\})$, $B\subset \chi^{-1}(\{0\})$. Furthermore, from the construction, we see that  $supp(\chi) \subset A \cup \partial A= \overline{A}$. We conclude that 
$\chi\in C_c^\infty(M)$. The function $\psi := \chi \cdot f$ belongs to $C_c^\infty(M)$ as well and furthermore 
$$||f - \psi||_\infty \leq  ||f|_K -\psi|_K||^{(K)}_\infty +   ||f|_{M\setminus K} -\psi|_{M\setminus K}||^{(M\setminus K)}_\infty$$
$$ = ||f|_K -f|_K||^{(K)}||_\infty + ||f\cdot (1-\chi)|_{M\setminus K}||^{(M\setminus K)}_\infty \leq 0 +   ||f|_{M\setminus K}||^{(M\setminus K)}_\infty = \varepsilon\:.$$
The proof is over since we have proved that if $f\in C_0(M) \cap C^\infty(M)$ and $\varepsilon>0$, then there exists $\psi \in C_c^\infty(M)$ such that $||f-\psi||_\infty < \varepsilon$.
 $\hfill\Box$\\

\noindent{\bf Proof of Lemma \ref{propL}}.
Noticing that $C_c^\infty(M)$ is dense in $L^2(M, \mu_g)$, let us first establish that $L_0|_{C_c^\infty(M)}$  is symmetric  in $L^2(M, \mu_g)$ -- where from now on $\mu_g$ is the volume form (a positive Borel measure) associated to the metric $g$.   Furthermore we also prove that   $-L_0|_{C_c^\infty(M)} \geq 0$.
\begin{lemma}\label{lemma33} With the hypotheses of Lemma \ref{propL}, (\ref{linkA}) in particular,
 $L_0|_{C_c^\infty(M)}$ is symmetric
and 
$-\langle h, L_0 h \rangle \geq 0$
if $h\in C_c^\infty(M)$.
\end{lemma}

\begin{proof}If $A$ is a vector field viewed as differential operator, taking advantage of a partition of the unit,  exploiting $Af = \nabla^{(g)}_Af = \sum_k A^j\nabla^{(g)}_j f$ and the fact that $\nabla^{(g)}_j|_{C_c^\infty(M)}$ is symmetric in $L^2(M, \mu_g)$, one immediately sees that, if $h,h' \in C_c^\infty(M)$,
$$\langle h', Ah \rangle = -\langle Ah', h \rangle -   \langle h',  (\nabla^{(g)} \cdot A) h \rangle \:,$$
where $ \nabla^{(g)} \cdot A$ acts as multiplicative operator.
Exploiting the fact that $C_c^\infty(M)$ is invariant under the action of $A_0$ and $A_i$ we find
$$\langle L_0 h', h\rangle =\langle h',   L_0h \rangle - 2\langle h', A_0h \rangle  + \sum_{i=1}^r\langle h', (\nabla^{(g)} \cdot A_i) A_ih \rangle
- \langle h', \nabla^{(g)} \cdot A_0   h\rangle $$
$$ +  \frac{1}{2}\sum_{i=1}^r \langle h', \left(\nabla^{(g)} \cdot (\nabla^{(g)} \cdot A_i) A_i\right) h\rangle  = \langle h', L_0 h \rangle$$
where we have used (\ref{linkA}) in the last passage. We have proved that $L_0|_{C^\infty_c(M)}$ is symmetric because $C^\infty_c(M)$ is dense and $\langle L_0 h', h\rangle =\langle  h', L_0h\rangle $ for all
$h,h' \in C_c^\infty(M)$.\\
 Regarding positivity, we have for $h\in C_0^\infty(M)$,
$$-\langle h, L_0h\rangle = -\frac{1}{2}\sum_{i=1}^r\int_M \overline{h}A_iA_i hd\mu_g - \int_M \overline{h} A_0h d\mu_g$$ 
$$=
\frac{1}{2}\sum_{i=1}^r\langle A_ih, A_ih\rangle+ \frac{1}{2}\sum_{i=1}^r\int_M (\overline{h}  \nabla^{(g)}\cdot A_i) A_i h d\mu_g
 - \int_M \overline{h} A_0h d\mu_g =  \frac{1}{2}\sum_{i=1}^r\langle A_ih, A_ih\rangle \geq 0$$
where we have used again (\ref{linkA}) in the last passage \end{proof}

Let us pass to prove that there is a solution $f\in C^\infty(M)$ of (\ref{EQR2})  when $h\in C_c^\infty(M)$.
Since $L_0|_{C_c^\infty(M)}$ is symmetric (``formally selfadjoint" in Shubin's terminology),  uniformly elliptic, and $C^\infty$-bounded, 
 Corollary 4.2 in \cite{shubin} implies that 
 $L_0|_{C_c^\infty(M)}$ is essentially selfadjoint in $L^2(M, \mu_g)$ and we will denote 
by $L'$ the unique selfadjoint extension of   $L_0|_{C_c^\infty(M)}$ (i.e., the closure of the latter in the Hilbert space $L^2(M, \mu_g)$).  Let us focus  on the equation for the unknown $f\in D(L')$
\begin{equation} L' f - \lambda f = h\:,\label{EQR}\end{equation}
when $h \in C_c^\infty(M)\subset L^2(M, \mu_g)$ and $\lambda >0$ are given. 
By  multiplying both sides with a test function $h'\in C_0^\infty(M)$ and integrating the result, using the fact that $L'$
is a selfadjoint extension of $L_0|_{C_c^\infty(M)}$, we find that an $f$ satisfying  (\ref{EQR}), if any,  must also satisfy (\ref{EQR2})
(where $L_0$ appears instead of $L'$!)
in {\em distributional sense},  since  $f \in D(L') \subset L^2(M, \mu_g) \subset {\cal D}'(M)$. 
 Elliptic regularity (Theorem 8.3.1 and Corollary 8.3.2 in \cite{hormander}) applied to the elliptic operator $A= L_0-\lambda I$ implies 
that, if $f$ exists, $f$ has  to belong to  $C^\infty(M)$ and also  satisfies (\ref{EQR2}) in classical sense.
As a matter of fact,  $f$  solving (\ref{EQR}) exists because every $\lambda>0$ belongs to the resolvent set of $L'$.
Indeed,   $-L'\geq 0$ (that is true because $-L'$ is the Hilbert-space closure of $-L_0|_{C_c^\infty(M)}$ which is positive for the lemma above)
 entails $\sigma(-L') \subset [0,+\infty)$.
A   solution of  (\ref{EQR}) (which also solves (\ref{EQR2}) and is smooth) therefore exists:
\begin{equation} f = R_\lambda(L')h\label{EQR3}\end{equation}
where $R_\lambda(L'):  L^2(M, \mu_g) \to L^2(M, \mu_g)$ is the resolvent operator of $L'$. \\
Let us pass to prove that $f \in C_0(M) \cap C_b^\infty(M)$ 
when $M$ is not compact (otherwise there is nothing to prove).  We henceforth assume that $M$ is non-compact.
We can say much more about $f$ in (\ref{EQR3}). First of all we observe that  the map ${\cal D}(M) = C_c^\infty(M) \ni h \mapsto R_\lambda(L') =f \in L^2(M; \mu_g)  \subset {\cal D}'(M)$ is sequentially continuous with respect to the natural topologies \cite{hormander} of $C_c^\infty(M)$ and ${\cal D}'(M)$ because $R_\lambda(L')$ is bounded in $L^2(M, \mu_g)$. Therefore we can apply Schwartz' kernel theorem \cite{hormander} that  establishes that there exists a distribution $G \in {\cal D}'(M\times M)$ such that, for every pair $h,h' \in C_c^\infty(M)$,
\begin{equation}\int_M  h'(x)\left(R_\lambda(L')h\right)(x) d\mu_g(x) = \int_{M \times M}   G(x,y) \:h'(x)h(y) \:d\mu_g(x)\otimes d\mu_g(y)\:.\label{intdist}\end{equation}
The integral on the left-hand side is a standard integral, the one on the right-hand side is  just a formal expression accounting for the action of a distribution.
However, Theorem 2.2  in \cite{shubin} (in the case $p=2$) proves that \\
(a) the distribution 
  $G$ is smooth  outside the diagonal, i.e., $G\in C^{\infty}(M\times M\setminus \Delta)$, where $\Delta = \{(x,x) \:|\: x \in M\}$, \\
(b)
there exists $\eta>0$ such that
 for every $\delta >0$ and every pair of multiindices $\alpha,\beta$, there exists $C_{\alpha, \beta, \delta}>0$  with
\begin{equation}
|\partial_x^\alpha \partial_y^\beta G(x,y)| \leq C_{\alpha, \beta, \delta} e^{-\eta d_g(x,y)}\quad \mbox{if $d_g(x,y) \geq\delta$,}\label{stima1}
\end{equation}
where 
$d_g$ is the geodesical distance on $(M,g)$ which is well defined since $M$ is connectedand the derivatives $\partial_x$ and $\partial_y$ are computed in a pair of Riemannian  charts (possibly the same).
Let us 
take $x_0 \not \in supp(h)$ and consider an open neighborhood $U$ of $x_0$ such that $\overline{U}$ is compact and $\overline{U}\cap supp(h) =\emptyset$. Since $U\times supp(h) \ni (x,y) \not \in \Delta$,
if $h' \in C_c^\infty(M)$ is supported in $U$  item (a) above permits us to  intepret litterally the integral on the right-hand side of (\ref{intdist}). Taking advantage of the Fubini theorem, we can rearrange (\ref{intdist})  to 
$$\int_M h'(x) \left(f(x) - \int_M G(x,y) h(y) d\mu_g(y)\right) d\mu_g(x) =0\:.$$
Since $C_c^\infty(U)$ is dense in $L^2(U, d\mu_g )$ and $x_0$ and  $U$ as above  are arbitrary, we can conclude that
\begin{equation}f(x) =  \int_M G(x,y) h(y) d\mu_g(y)\quad \mbox{almost everywhere if $x\not \in supp(h)$.}\label{fintG}\end{equation}
This result can be made even stronger observing that the function $\overline{U}\times supp(h) \ni (x,y) \mapsto G(x,y)h(y)$ is smooth due (a) and thus continuous and bounded. Hence,  a direct use of dominated convergence theorem proves that $$U\ni x \mapsto   \int_M G(x,y) h(y) d\mu_g(y)$$ is continuous as well. Since the left-hand side of (\ref{fintG}) is also continuous, we have  proved that
\begin{equation}f(x) =  \int_M G(x,y) h(y) d\mu_g(y)\quad \mbox{if $x \in M \setminus supp(h)$.}\label{fintG2}\end{equation}
Let us conclude the proof by establishing  that $f$ vanishes at infinity and $||A_kf||_\infty < +\infty$ for $k=0,1,\ldots,r$.
Since $supp(h)$ is compact and the open geodesical balls are a basis of the topology of $M$, there is a finite covering
$\{B_{r_n}(x_n)\}_{n=1,\ldots, N}$ of $supp(h)$ made 
 of closed geodesical balls with finite radius. As a consequence there exist a sufficiently large closed ball $B_R(x_0)$ including $supp(f)$. (It is sufficient to enlarge the radius  $r_0$ of 
$B_{r_0}(x_0)$, to $R:= D+P$ where $D:= \max\{d_g(x_0,x_n)\:|\:  n=0,1, \ldots, N\}$ and $P= \max\{r_n\:|\: n=0,1, \ldots, N\}$.) Notice that  for every closed ball $B_R(x_0)$, with arbitary large $R>0$,  it must hold
$M \setminus B_R(x_0) \neq \emptyset$ necessarily, otherwise $M$ would be compact due to Lemma \ref{lemmaC}
since $M$ is of bounded geometry, and $M$ is not compact by hypothesis. With $\eta>0$ as in (b), choose $\delta>0$ and define  another closed ball $B_{R'}(x_0)$  with $R' > \delta + R$. If $y \in B_R(x_0)$ and $x \in M \setminus B_{R'}(x_0)$ we have
$d_g(x,y) \geq d_g(x,x_0) - R > R'-R > \delta +R -R > \delta$ so that we can use the inequality (\ref{stima1}) with $\alpha=\beta=0$, finding
\begin{equation} |f(x)| \leq  \int_M |G(x,y)|  |h(y)| d\mu_g(y) \leq vol_g(B_R(x_0)) C_\delta ||h||_\infty e^{\eta R} e^{-\eta d_g(x,x_0)}\quad \mbox{if $x \in M \setminus B_{R'}(x_0)$}\label{stima3}\end{equation}
where, for $x \in M \setminus B_{R'}(x_0)$ and $y\in B_R(x_0)$, we took advantage of
$$R+ d_g(x,y) \geq d_g(x,x_0)$$
so that
$$-\eta d_g(x,y) \leq -\eta d_g(x,x_0) + \eta R$$
which implies (\ref{stima3}) through (\ref{stima1}).
To conclude,
with $h, x_0,\eta,\delta, R, R', C_\delta$ fixed as above  and if $||h||_\infty >0$ (otherwise there is nothing to prove since $f=0$), 
for every  $\varepsilon>0$ define
$$R_\varepsilon :=   - \frac{1}{\eta}\log \left( \frac{\varepsilon}{vol_g(B_R(x_0)) C_\delta ||h||_\infty e^{\eta R}}  \right).$$
For every  $\varepsilon>0$ (such small that $R_\varepsilon > R'$),
consider the closed ball $B_{R_\varepsilon}(x_0)$ which is compact in view of Lemma \ref{lemmaC}.
Here,  (\ref{stima3}) yields 
\begin{equation}|f(x)|\leq  vol_g(B_R(x_0)) C_\delta ||h||_\infty e^{\eta R} e^{-\eta R_\varepsilon}    =  \varepsilon\quad  \mbox{if $x \in M \setminus B_{R_\varepsilon}(x_0)$.}\label{stima3333}\end{equation} We have proved that $f \in C_0(M)$.
With a procedure similar to the we used to prove (\ref{fintG2}) based on Lagrange theorem and dominated convergence theorem  proves that in every Riemannian coordinate patch,
\begin{equation}\partial_{x}^\alpha f(x) =  \int_M \partial_{x}^\alpha G(x,y) h(y) d\mu_g(y)\quad \mbox{if $x \in M \setminus supp(h)$.}\label{fintG22}\end{equation}
Every $\partial_{x}^\alpha f$ is necessarily bounded on a finite covering of Riemannian charts of a compact ball $B_{R_\epsilon}$ including $supp(h)$.  Outside $B_{R_\epsilon}$, a procedure similar to that followed to prove (\ref{stima3333}) and relying on (\ref{stima1}) for $\beta=0$  proves that there is a constant  $H_\alpha<+\infty$ such that, in every local  Riemannian coordinate patch on $M$ and for $i=1,\ldots,d$,
\begin{equation}
|\partial_{x}^\alpha f(x)| < H_\alpha\:.\label{fintG23}
\end{equation} We have established that $f \in C_b^\infty(M)$ concluding the proof.  $\hfill \Box$\\

\noindent{\bf Proof of Lemma \ref{lemmaMN}}.
Let us consider $u \in D(M)$ and the map $u(t) := e^{tM}u$ for $t\in [0,+\infty)$. Due to Proposition \ref{ACPsol} (i.e. Proposition 6.2 in \cite{EN1}) $u(t) \in D(M)$ and this map is the unique classical solution of the Cauchy problem associated to $M$ with initial datum $u$. In particular it is continuously differentiable and satisfies $\frac{du}{dt} = Mu(t)$.
Since $M\subset N$, it also satisfies $\frac{du}{dt} = Nu(t)$ and thus, again for Proposition \ref{ACPsol}, it is also the unique solution of the Cauchy problem associated to $N$ with initial datum $u$. That is  $u(t) = e^{tN}u$. We have in particular found that, if $u\in D(M)$, then $e^{tN}u \in D(M)$ for $t\in [0,+\infty)$, so that $D(M)$ is invariant under the semigroup generated by $N$. Proposition 6.2 in \cite{EN1} implies that $D(M)$ is a core for $N$. Since $M\subset N$
and both operators are closed, then $M=N$. $\hfill \Box$\\

\noindent{\bf Proof of Lemma \ref{propL2}}.  Let us denote by  $L''$  the {\em Hilbert-space} closure $\overline{L_0|_{C_c^\infty(M)}}$. We remark that $L_0|_{C_c^\infty(M)}$ is closable since its adjoint has a dense domain, as one can easely prove by a integration-by-parts argument. We write $L''$ in place of $L'$, to stress that the differential  operator $L_0$ 
whose $L''$ is the Hilbert space closure over the domain $C_c^\infty(M)$ now
 includes the perturbation $B$.  The proof, except for a point, is  identical to that of  proposition \ref{teoL} 
using Proposition 4.1  in place of its Corollary 4.2 in \cite{shubin}, observing that elliptic regularity
 works also for $-L''$ since this property  depends only 
on the second order part of $L_0$,  and noticing that the properties  of $G$ established in 
Theorem 2.2 of \cite{shubin}, (\ref{stima1}) in particular, are valid also if $L_0|_{C_c^\infty(M)}$ is not symmetric. The only  new item to prove separately is that
there is a $\lambda>0$ in the resolvent set of $-L''$, which, differently from $-L'$,  is no longer positive and
  selfadjoint  due to the presence of the term $B$. With this result the proof of the thesis concludes.
 Let us prove the existence of such $\lambda >0$ by establishing that $L''$ is the generator of a strongly
continuous semigroup 
in $L^2(M, \mu_g)$ under the hypotheseis (\ref{dominance}): in this case, the standard spectral bound  of generators of strongly continuos 
semigroups  (Corollary 1.13 in \cite{EN1})  implies  that $Re(\sigma(L''))$ has finite upper
 bound so that $\rho(L'') \cap (0,+\infty) \neq \emptyset$ and the requested $\lambda>0$ exists.   In the rest of the proof $-L'$
 will denote again the positive selfadjoint operator 
used in the proof of proposition \ref{teoL}, which is the Hilbert-space closure of    $L_0|_{C_c^\infty(M)}$, where $A_0$ does {\em not}
 contain the perturbation 
$B$. As is known from Proposition 4.1 in \cite{shubin}, $D(L'') = D(L')= W^2_2(M)$ (see \cite{shubin} for the definition of those
 Sobolev spaces on smooth Riemannian manifolds of  bounded geometry).  The operator $B|_{C_c^\infty(M)}$ 
is $L^2(M,\mu_g)$-closable since its adjoint has dense domain (it including $C_c^\infty(M)$) and the closure of $B|_{C_c^\infty(M)}$ has  domain that evidently  includes
$W^2_2(M)$ because $C^\infty_c(M)$ is dense in $W^2_1(M) \supset W^2_2(M)$ \cite{shubin}. We intend to prove that, defining $L'+ \overline{B|_{C_c^\infty(M)}}$ on the domain $W^2_2(M)$ of the first addend, 
then $L'+ \overline{B|_{C_c^\infty(M)}}$ is (i) closed and (ii) it is the generator of a strongly continuous semigroup. Notice that, in this case 
$L'+ \overline{B|_{C_c^\infty(M)}}= L''$ since $L'' \subset L'+ \overline{B|_{C_c^\infty(M)}}$ by construction ($L''$ is the closure of $L_0|_{C_0^\infty}$
whereas the right-hand side is a closed extension of that)
 and the two sides of the 
inclusion have the same domain $W^2_2(M)$. Hence (i) and (ii) imply that $L''$ itself is the generator of a strongly continuous
 semigroup as wanted. To conclude the proof, we prove that (i) and (ii) are true if (\ref{dominance}) holds. 
 Since $\sigma(L')\subset (-\infty,0]$
and $L'$ is selfadjoint, $\{e^{tL'}\}_{t \in [0,+\infty)}$ is an analytic semigroup in $L^2(M, \mu_g)$. To prove (i) and (ii), according to 
Theorem X.54 in \cite{RS2}, it is sufficient to demonstrate that for every $a>0$, there is a corresponding $b>0$ such that (the norm is that of 
$L_2(M, \mu_g)$)
$$||\overline{B|_{C_c^\infty(M)}} \psi || \leq a|| L' \psi|| + b||\psi||\quad \mbox{for all $\psi \in W_2^2(M)$.}$$
Observe that, since $C_c^\infty(M)$ is a core for $L'$(it is essentially selfadjoint thereon) 
 and $\overline{B|_{C_c^\infty(M)}}$ is closed, the condition above is equivalent to
$$||B \psi || \leq a|| L' \psi|| + b||\psi||\quad \mbox{for all $\psi \in C_c^\infty(M)$.}$$
In turn, according to the remark on the condition (iii) on p. 162 of \cite{RS2}, the condition above is equivalent to the next statement:
For every $a>0$ there is $b>0$ such that 
 \begin{equation} ||B \psi ||^2 \leq a|| L' \psi||^2 + b||\psi||^2\quad \mbox{for all $\psi \in C_c^\infty(M)$}\label{lastCC}
\end{equation}
(where these $a,b$ are generally different from those in the previous inequality).
To conclude we prove that (\ref{lastCC}) is consequence of (\ref{dominance}). From the latter, replacing $\xi_k$ with $\nabla^{(g)}_k\psi$, 
if $\psi \in C^\infty_c(M)$, we have
$$\int_M \overline{(B\psi)(x)} (B\psi)(x) d\mu_g(x) \leq c\int_M
 \sum_{i=1}^r \sum_{a,b=1}^d\overline{(A^a_i \nabla^{(g)}_a \psi)(x)}
 (A^b_i \nabla^{(g)}_b \psi)(x) d\mu_g(x)$$ $$=-2c\int_M \overline{\psi(x)}(L'\psi)(x) d\mu_g(x)\:. $$
Namely, if $\langle\cdot,\cdot\rangle$ is the scalar product in $L^2(M,\mu_g)$, standard results of spectral theory
\cite{Moretti,Schm} yield
$$||B\psi||^2 \leq 2c\langle \psi, -L'\psi \rangle =2c\int_{\bR^+} \lambda d\nu_{\psi}(\lambda) $$
where $\nu_\psi(E) = \langle \psi, P^{(-L')}(E)\psi\rangle$,  with $P^{(-L')}$ being is the {\em spectral measure}  of the selfadjoint positive operator 
$-L'$ and $E\subset \bR$ any Borel set. Here observe that, since $c>0$, for every $a>0$ there is $b>0$ such that 
$$2c\lambda \leq a\lambda^2 + b\quad \mbox{for all  $\lambda\geq 0$.}$$
It is in fact sufficient to use $b= c^2/a$.
Therefore, again from standard results of spectral theory,
$$||B\psi||^2 \leq 2c\int_{\bR^+} \lambda d\nu_{\psi}(\lambda)  \leq a\int_{\bR^+} \lambda^2 d\nu_{\psi}(\lambda) + b\int_{\bR^+}1\: d\nu_{\psi}(\lambda)=
 a||-L'\psi||^2 + b||\psi||^2\:. $$
In summary, for every $a>0$, there is $b>0$ such that (\ref{lastCC}) holds
$$ ||B \psi ||^2 \leq a|| L' \psi||^2 + b||\psi||^2\quad \mbox{for all $\psi \in C_c^\infty(M)$,}$$
concluding the proof.
$\hfill \Box$\\

\noindent {\bf Proof of Proposition \ref{teo2}}.

(a) Let us start with a given $r \in (0, I_{(M,g)})$ and
consider a Riemannian system of coordinates in the ball $B^{(M,g)}_r(p)$.
Expanding $g_{ab}(y)$ around $0$ up to the first order with the usual Taylor expansion,  we have
$$g_{ab}(y)=\delta_{ab} + 0 + R^{(2)}_{ab}(y)$$
where, for some $\xi \in B_r(0)$,
$$ R^{(2)}_{ab}(y)= \frac{1}{2!} \sum_{i,j}\frac{\partial^2g_{ab}}{\partial y^i\partial y^j}|_\xi y^iy^j \quad  y \in B_r(0), \quad i,j=1,
\ldots, d\:.$$
Taking  the second bound in (\ref{estimateg}) into account for $k=2$ and using $|y^k|\leq r$ we have
$$\left|||A(y(q))||^2 - ||A(y(q))||^2_g \right|= \left|\sum_{a,b=1}^d A^a(y) g_{ab}(y) A^b(y) -A^a(y)\delta_{ab} A^b(y)\right|  = \left|\sum_{a,b=1}^dA^a(y)  A^b(y) R^{(2)}_{ab}(y)\right|$$
$$\leq \sum_{a,b=1}^d |A^a(y)|  |A^b(y)| \frac{1}{2}C^{(r)}_2d^2r^2\leq \frac{C^{(r)}_2 d^2r^2}{2}\sum_{i,j=1}^d \|A(y)\|  \|A(y)\|  = \frac{C^{(r)}_2d^4 r^2}{2}||A(y)||^2\:.$$
In particular 
$$||A(y(q))||^2 - ||A(y(q))||^2_g  \leq \frac{C^{(r)}_2d^4 r^2}{2}||A(y(q))||^2$$
namely, if $||y||<r$, we have,
$$ \left(1-\frac{C^{(r)}_2d^4 r^2}{2}\right)||A(y)||^2 \leq ||A(y(q))||^2_g \:.$$
Restricting $r$ to $r_0>0$ such that\footnote{It is always possible to find such $r_0$ since the functions $r\mapsto C_k^{(r)}$ are monotone not-decreasing.} $(1-d^4 r_0^2C^{(r_0)}_2/2) >0$ and defining $k_1 :=  (1-d^4 r_0^2C^{(r)}_2/2) ^{-1}$, we conclude that (a) is valid for 
$y \in B_{r_0}(0)$, i.e., $q \in B^{(M,g)}_{r_0}(p)$.\\
(b) Let us first show that, if $r_0>0$ is suitably small, then 
 \begin{equation}\label{quasib} ||T(y(q))||^2 \leq k_2 ||T(q)||^2_g\:, \quad\mbox{for all}\:\:  q \in B^{(M,g)}_{r_0}(p)\end{equation}
for some $k_2 \geq 0$ independent of $T$ and $p$, for every smooth tensor field $T$ of order $(1,1)$. 
The proof is strictly analogous to that of (a), observing that
\begin{equation}\label{quasib2} ||T(y(q))||^2 - ||T(y(q))||_g^2 = \sum_{a,b,i,j=1}^d T_a^i(y) \left(\delta^{ab}\delta_{ij} 
- g^{ab}(y)g_{ij}(y)\right)T_b^j(y)\end{equation}
and
$$ g^{ab}(y)g_{ij}(y) =  \delta^{ab}\delta_{ij}+0 + R^{(2)ab}_{ij}(y)$$
where, for some $\xi \in B_r(0)$,
$$R^{(2)ab}_{ij}(y)= \frac{1}{2!} \sum_{i,j}\frac{\partial^2 g^{ab}g_{ij}}{\partial y^i\partial y^j}|_\xi y^iy^j \quad  y \in B_r(0), \quad i,j=1,
\ldots, d\:,$$
Using in (\ref{quasib2}) both the second bound in (\ref{estimateg}) and (\ref{estimateg2}) for $k=0,1$ as we did  in the proof  (a) we obtain (\ref{quasib}).
To conclude the proof of (b), observe that, if $y \in B_{r_0}(0)$,
$$\partial_{y^a}A^i = (\nabla^{(g)}_a A)^i - \sum_{c=1}^d \Gamma^i_{ac} A^c$$
so that, using  (\ref{estimateg3}) toghether with rough estimates $|A^i| \leq ||A||$, $|\nabla^{(g)}_a A^i| \leq \|\nabla^{(g)} A\|$, we have
$$\|\nabla A\|^2 \leq \| \nabla^{(g)} A\|^2 + 2d^3  J_{0}^{(r_0)} \|A\| \|\nabla^{(g)} A\| + d^4 (J_{0}^{(r_0)})^2 \|A\|^{2}.$$
 Finally observe that (a) and (\ref{quasib}) respectively imply
$$\|A\| \leq k_1\|A\|_g\quad \mbox{and}\quad  \|\nabla^{(g)} A\| \leq \sqrt{k_2}  \| \nabla^{(g)} A \|_g$$
which, inserted in the previous inequality, yield
$$\|\nabla A(y(q))\|^2 \leq k_2\| \nabla^{(g)} A(q)\|_g^2 + 2d^3  J_{0}^{(r_0)} k_1\sqrt {k_2} \|A(q)\|_g \|\nabla^{(g)} A(q)\|_g + d^4 (J_{0}^{(r_0)})^2
 k_1^2\|A(q)\|^2_g$$
which must hold if $q \in B^{(M,g)}_{r_0}(p)$. By construction, the constants, $k_1$, $k_2$, $k_3 := d^4(J_{0}^{(r_0)})^2k_1^2$, and  $k_4 := 2d^3 J_{0}^{(r_0)} k_1\sqrt {k_2}$ do not depend on $A$ and the estimate is valid for every $p\in M$ provided $q \in B^{(M,g)}_{r_0}(p)$.  $\hfill \Box$\\

\end{document}